\documentclass{amsart}

\usepackage{amssymb,amsmath,accents}
\usepackage{amsfonts,amsthm,mathrsfs}
\usepackage{cases,empheq} 
\usepackage{graphicx} 
\usepackage{subcaption}
\usepackage{diagbox}

\numberwithin{equation}{section}
\newtheorem{thm}{Theorem}[section]
\newtheorem{cor}[thm]{Corollary}
\newtheorem{lemma}[thm]{Lemma}

\newtheorem{remark}[thm]{Remark}

\allowdisplaybreaks

\makeatletter
\@namedef{subjclassname@2020}{\textup{2020} Mathematics Subject Classification}
\makeatother

\begin{document}
\title[A singular limit of the Kobayashi-Warren-Carter system]{Fractional time differential equations as a singular limit of the Kobayashi--Warren--Carter system}

\author[Y.~Giga]{Yoshikazu Giga}
\address[Y.~Giga]{Graduate School of Mathematical Sciences, The University of Tokyo, 3-8-1 Komaba, Meguro-ku, Tokyo 153-8914, Japan.}
\email{labgiga@ms.u-tokyo.ac.jp.}

\author[A.~Kubo]{Ayato Kubo}
\address[A.~Kubo]{Department of Mathematics, Faculty of Science, Hokkaido University, Kita 10, Nishi 8, Kita-Ku, Sapporo, Hokkaido, 060-0810, Japan.}
\email{kubo.ayato.j8@elms.hokudai.ac.jp}

\author[H.~Kuroda]{Hirotoshi Kuroda}
\address[H.~Kuroda]{Department of Mathematics, Faculty of Science, Hokkaido University, Kita 10, Nishi 8, Kita-Ku, Sapporo, Hokkaido, 060-0810, Japan.}
\email{kuro@math.sci.hokudai.ac.jp}

\author[J.~Okamoto]{Jun Okamoto}
\address[J.~Okamoto]{Institute for the Advanced Study of Human Biology, Kyoto University, Yoshida-Konoe-Cho, Sakyo-ku, Kyoto 606-8501, Japan.}
\email{okamoto.jun.8n@kyoto-u.ac.jp}

\author[K.~Sakakibara]{Koya Sakakibara}
\address[K.~Sakakibara]{Faculty of Mathematics and Physics, Institute of Science and Engineering, Kanazawa University, Kakuma-machi, Kanazawa-shi, Ishikawa 920-1192, Japan; 
RIKEN iTHEMS, 2-1 Hirosawa, Wako-shi, Saitama 351-0198, Japan.}
\email{ksakaki@se.kanazawa-u.ac.jp}

\author[M.~Uesaka]{Masaaki Uesaka}
\address[M.~Uesaka]{Graduate School of Mathematical Sciences, The University of Tokyo, 3-8-1 Komaba, Meguro-ku, Tokyo 153-8914, Japan;
Arithmer Inc., ONEST Hongo Square 3F, 1-24-1 Hongo, Bunkyo-ku, Tokyo 113-0033, Japan.}


\begin{abstract}
	This paper is concerned with a singular limit of the Kobayashi--Warren--Carter system, a phase field system modelling the evolutions of structures of grains.
	Under a suitable scaling, the limit system is formally derived when the interface thickness parameter tends to zero.
 	Different from many other problems, it turns out that the limit system is a system involving fractional time derivatives, although the original system is a simple gradient flow.
 	A rigorous derivation is given when the problem is reduced to a gradient flow of a single-well Modica--Mortola functional in a one-dimensional setting.
\end{abstract}

\maketitle


\section{Introduction} \label{S1} 

We consider the Kobayashi--Warren--Carter system, introduced in \cite{KWC1,KWC2,KWC3}, to model evolutions of structures in a multi-grain problem.
 It is a kind of phase-field system in a domain $\Omega$ in $\mathbf{R}^n$, formally a gradient flow of the energy
\begin{align}
	&E_\mathrm{KWC}^\varepsilon(u,v) := \int_\Omega \alpha_0(v) |\nabla u| + E_\mathrm{sMM}^\varepsilon(v),  \label{EEKWC} \\
	&E_\mathrm{sMM}^\varepsilon(v) := \int_\Omega \frac\varepsilon2 |\nabla v|^2\, dx + \int_\Omega \frac{1}{2\varepsilon} F(v)\, dx. \label{EEsMM}
\end{align}
Here, $\alpha_0(v)\geq0$ is a given function, typically $\alpha_0(v)=sv^2$, with a constant $s>0$, and $F(v)$ is a single-well potential, typically $F(v)=a^2(v-1)^2$, with a constant $a>0$.
 The functional $E_\mathrm{sMM}^\varepsilon$ is often called a single-well Modica--Mortola functional.
 The Kobayashi--Warren--Carter system is regarded as a gradient flow with respect to $L^2$-inner product
\begin{multline*}
	\left((u_1,v_1),(u_2,v_2)\right)
	= \int_\Omega \alpha_w u_1 u_2 \, dx
	+ \tau \int_\Omega v_1 v_2 \, dx,\\
	(u_i, v_i) \in L^2(\Omega) \times L^2(\Omega), \quad
	i=1,2,
\end{multline*}
where $\alpha_w\geq0$ and $\tau>0$ are weights.
 The function $\alpha_w=\alpha_w(v)$ is given, but it may depend on $v\in L^2(\Omega)$, so the above inner product is a Riemann metric on the tangent bundle $TL^2(\Omega)$.
 A typical form of $\alpha_w(v)$ equals $\alpha_w(v)=\tau_0v^2$, where $\tau_0$ is a positive constant.
 We consider the gradient flow of $E_\mathrm{KWC}^\varepsilon$ under this metric, and its explicit form is
\begin{empheq}[left = {\empheqlbrace \,}]{alignat = 2}
    \tau v_t &= \varepsilon\Delta v - \frac{1}{2\varepsilon} F'(v) - \alpha'_0(v)|\nabla u|, \label{EG1}  \\
    \alpha_w(v) u_t &= \operatorname{div} \left( \alpha_0(v) \frac{\nabla u}{|\nabla u|} \right). \label{EG2}
\end{empheq}
An explicit form in \cite{KWC1} corresponds to the case when $s=\varepsilon^2$, $\tau_1=\tau\varepsilon$, $F(v)=(v-1)^2$, $\alpha_0(v)=\varepsilon v^2$, $\alpha_w(v)=\tau_0 v^2/\varepsilon$ with $\tau_0>0$.
 In other words,
\begin{empheq}[left = {\empheqlbrace \,}]{alignat = 2}
    \tau_1 v_t &= s\Delta v + (1-v) -2sv|\nabla u|, \label{EKWC1}  \\
    \tau_0 v^2 u_t &= s\operatorname{div} \left( v^2 \frac{\nabla u}{|\nabla u|} \right). \label{EKWC2}
\end{empheq}
The function $v$ represents an order parameter.
 Where $v=1$ corresponds to a grain region, and where $v$ is away from $1$ corresponds to grain boundaries.
 The function $u$ represents a structure-like averaged angle in each grain.

We are interested in a singular limit problem for \eqref{EG1}, \eqref{EG2} as $\varepsilon\downarrow0$.
 It turns out that the correct scaling of time should be $\tau=\tau_1/\varepsilon$, while $\tau_1$ is independent of $\varepsilon$.
 Since the system \eqref{EG1}--\eqref{EG2} is regarded as a gradient flow of $E_\mathrm{KWC}^\varepsilon(u,v)$ of \eqref{EEKWC}, we are tempting to expect that the limit flow is the gradient flow of its limit energy $E_\mathrm{KWC}^0$ which was obtained in our papers \cite{GOU,GOSU}.
 Surprisingly, this conjecture is wrong.
 The limit flow contains a fractional time derivative.
 In this paper, we consider the problem in a one-dimensional setting.
 Moreover, we consider a special but typical case when the problem is essentially reduced to a single equation for $v$ in \eqref{EG1} because handing \eqref{EG2} is technically involved since it is a total variation flow type equation.
 This reduced problem becomes a linear problem and is easy to discuss.

We consider \eqref{EG1}--\eqref{EG2}, where $\Omega$ is an interval $\Omega=(-L,L)$ and impose the Dirichlet boundary condition for $u$ and the Neumann boundary condition for $v$.
 More precisely,
\begin{equation} \label{EDi}
	u(-L,t) = 0, \quad u(L,t) = b > 0
\end{equation}
while
\begin{equation} \label{ENe}
	v_x(\pm L,t) = 0 \quad\text{for}\quad t > 0.
\end{equation}
We set
\begin{equation} \label{EEST}
	F(v) = a^2(v-1)^2 \quad\text{with}\quad a \geq 0, 
	\quad \alpha_0(v) = v^2.
\end{equation}
We expect that the function
\begin{eqnarray*}
	u^b(x) = \left\{
\begin{array}{ll}
	b, & x > 0, \\
	0, & x < 0
\end{array}
\right.
\end{eqnarray*}
with $b>0$ solves \eqref{EG2}.
 Since equation \eqref{EG2} is of total variation flow type, the definition of a solution is not obvious.
 Fortunately, under a suitable assumption of $v$, the function $u^b$ solves \eqref{EG2} under \eqref{EDi}, as shown in the following lemma by setting $\beta=\alpha_0(v)$.
\begin{lemma}[A stationary solution] \label{LST}
Assume that $\beta\in C[-L,L]$ satisfies
\[
	\beta(0) \leq \beta(x) \quad\text{for all}\quad
	x \in (-L,L).
\]
Then $u^b$ solves
\[
	\left( \beta \frac{u_x}{|u_x|} \right)_x = 0
	\quad\text{in}\quad (-L,L)
\]
under \eqref{EDi}.
\end{lemma}

We stress that the notion of a solution of the equation for $u$ in Lemma \ref{LST} and also \eqref{EG2} under the Dirichlet condition \eqref{EDi} is not obvious and will be discussed in Section \ref{S4}.

Problem \eqref{EG1}--\eqref{EG2} is reduced to
\begin{equation} \label{ER}
	\frac{\tau_1}{\varepsilon} v_t = \varepsilon v_{xx} - \frac{a^2(v-1)}{\varepsilon} - 2b\partial_x(1_{x>0})v \quad\text{in}\quad \Omega \times (0,\infty)
\end{equation}
under the boundary condition
\begin{equation} \label{ERB}
	v_x(\pm L, t) = 0, \quad t > 0
\end{equation}
and the initial condition
\begin{equation} \label{ERI}
	v(x,0) = v_0^\varepsilon(x), \quad x \in \Omega,
\end{equation}
where $1_{x>0}$ is a characteristic function of $(0,+\infty)$, i.e., the Heaviside function, so that its distributional derivative equals the Dirac $\delta$ function.
 By a scaling transformation $y=x/\varepsilon$, \eqref{ER} becomes
\begin{equation} \label{ERR}
	\tau_1 V_t = V_{yy} - a^2(V-1) - 2b\partial_y(1_{y>0})V
\end{equation}
in $(-L/\varepsilon,L/\varepsilon)\times(0,\infty)$ for $V=V^\varepsilon(y,t)=v^\varepsilon(\varepsilon y,t)$, where $v^\varepsilon$ is a solution of \eqref{ER}.
 Thus, we expect this limit $V=\lim_{\varepsilon\to0}V^\varepsilon$ solves \eqref{ERR} on $\mathbf{R}$ and is bounded.
 Since the solution $v^\varepsilon$ of \eqref{ER} is expected to converge to $1$ except at $x=0$, we are interested in the behaviour of $\xi^\varepsilon(t)=v^\varepsilon(0,t)$.
 More precisely, we would like to find the equation which $\xi=\lim_{\varepsilon\to0}\xi^\varepsilon$ solves.
 Let $E_\mathrm{sMM}^{0,b}(\zeta)$ be a limit energy obtained by \cite{GOU} for a set-valued function $\Xi$, defined as $\Xi(x)=\{1\}$ for $x\neq0$ and $\Xi(0)=[\zeta,1]$, $\zeta\in(0,1)$.
 In other words,
\begin{align*}
	E_\mathrm{sMM}^{0,b}(\zeta) &= b\zeta^2 + 2G(\zeta)= b\zeta^2 + a(\zeta-1)^2,\quad\text{where}\quad G(\zeta) = \int_1^\zeta \sqrt{F(\rho)}\, d\rho.
\end{align*}
A key observation is to derive an equation for $\xi$.
 For $\alpha\in\mathbf{R}$, $\beta>0$, we set
\[
	f_\beta^\alpha(t) := \frac{e^{-\alpha t} t^{\beta-1}}{\Gamma(\beta)}.
\]
We consider well-prepared initial data in the sense that it is continuous and solves \eqref{ERR} outside $y=0$.
\begin{lemma}[Limit equation] \label{LLE}
Assume that $\tau_1=1$.
 Let $V$ be the bounded solution of \eqref{ERR} in $\mathbf{R}\times(0,\infty)$ with well-prepared initial data $V(y,0)=1-ce^{-a|y|}$ with some $c\in\mathbf{R}$.
 Then $\xi(t)=V(0,t)$ solves
\begin{equation} \label{ELE}
	\int_0^t m_a(t-s) \xi_s(s)\, ds = -\operatorname{grad} E_\mathrm{sMM}^{0,b}(\xi)
\end{equation}
with
\[
	m_a(t) = 2\left\{f_{1/2}^{a^2}(t) + a^2\int_0^t f_{1/2}^{a^2}(s)\, ds  - a \right\}.
\]
Moreover, $m_a'(t)<0$ and $m_a(t)>0$, and $\lim_{t\to\infty}m_a(t)=0$.
\end{lemma}
The assumption $\tau_1=1$ is just for the convenience of presentation.
 The formula for general $\tau_1$ is obtained by rescaling the time variable $t$ by $t'=t\tau_1$.

Note that in case $a=0$, the left-hand side of \eqref{ELE} becomes the Caputo derivative $\partial_t^{1/2}$.
 For even initial data, the equation \eqref{ERR} in $\mathbf{R}\times(0,\infty)$ is reduced to the Robin boundary problem in $(-\infty,0)\times(0,\infty)$ with
\[
	\partial_y V(0,t) + bV(0,t) = 0.
\]
The Caputo derivative appears in the equation of the boundary value.
\begin{cor} \label{CCap}
Let $w=w(x,t)$ be the bounded solution of the heat equation
\[
	w_t - w_{xx} = 0
	\quad\text{in}\quad (-\infty,0) \times (0,\infty)
\]
with the Robin boundary condition
\[
	w_x(0,t) + bw(0,t) = 0
	\quad\text{for}\quad t > 0.
\]
Assume that $w(x,0)=-c$ for some $c\in\mathbf{R}$; the boundary value $\xi(t)=w(0,t)$ solves
\[
	\int_0^t f_{1/2}^0 (t-s) \xi_s(s)\,ds
	= -b\xi(t), \quad t > 0.
\]
In other words, $\partial_t^{1/2}\xi=-b\xi$, where $\partial_t^{1/2}$ is the Caputo half derivative.
\end{cor}
It is well known that the fractional Laplace operator $(-\Delta)^{1/2}$ arises as the Dirichlet--Neumann map of the Laplace equation.
 Here, the Caputo derivative $\partial_t^{1/2}$ is obtained as the Dirichlet--Neumann map of the heat equation.
 Formally, it is easy to guess since the Robin boundary condition yields $\partial_t^{1/2}\xi+b\xi=0$ by replacing $\partial_x$ with $\partial_t^{1/2}$, which is natural since $\partial_t=\partial_x^2$ for $w$.
 In a seminal paper, Caffarelli and Silvestre \cite{CS} show that $(-\Delta)^\gamma$ ($0<\gamma<1$) is obtained as the Dirichlet--Neumann map for the degenerate Laplace equation.
 We remark that $\partial_t^\gamma$ is obtained as the Dirichlet--Neumann map of a degenerate heat equation $w_t-(-x)^\alpha w_{xx}=0$ with $\gamma=1/(2-\alpha)$ as in \cite{CS}; see Remark \ref{RCS} at the end of Section \ref{S2}.
 We do not pursue this problem in this paper. 
 
Since the equation is linear and of constant coefficients, we can use the Laplace transform to obtain the desired equation.
 As we expect, proving the convergence of $V^\varepsilon$ and $v^\varepsilon$ is not tricky.
\begin{lemma}[Convergence] \label{LCon}
Let $V$ be the bounded solution of \eqref{ERR} in $\mathbf{R}\times(0,\infty)$ with initial data $V_0\in C(\mathbf{R})\cap L^\infty(\mathbf{R})$.
 Let $v^\varepsilon$ be the solution of \eqref{ER} under \eqref{ERB} and \eqref{ERI}.
 Assume that $V^\varepsilon(x,0)\in C\left[-L/\varepsilon,L/\varepsilon\right]$ converges to $V_0$ in the sense that
\[
	\lim_{\varepsilon\downarrow 0} \sup_{|x|\leq L/\varepsilon} \left| V^\varepsilon(x,0) - V_0(x) \right| = 0.
\]
 Then $V^\varepsilon$ converges to $V$ locally uniformly in $\mathbf{R}\times[0,\infty)$.
 In particular, $v^\varepsilon$ converges to $1$ locally uniformly in $\left(\mathbf{R}\backslash\{0\}\right)\times[0,\infty)$, and $\xi^\varepsilon$ converges to $\xi(t)=V(0,t)$ locally uniformly in $[0,\infty)$.
\end{lemma}
Applying Lemma \ref{LLE} and Lemma \ref{LCon}, we can obtain a characterisation of the limit equation.
\begin{thm} \label{Tmain}
Assume that $\tau_1=1$.
 Let $v^\varepsilon$ be the solution of \eqref{ER} under \eqref{ERB} and \eqref{ERI}.
 Assume that $v_0^\varepsilon$ is well-prepared in the sense
\[
	\sup_{x\in\Omega}\left|v_0^\varepsilon(x)-\left(1-ce^{-a|x|/\varepsilon}\right)\right|\rightarrow0
\]
as $\varepsilon\to0$ with some $c$ independent of $\varepsilon$.
 Then $\xi^\varepsilon(t)$ converges to $\xi$ locally uniformly in $[0,\infty)$, and $\xi$ solves \eqref{ELE}.
 Moreover, the graph of $v_\varepsilon$ converges to a set-valued function $\Xi$ of the form
\begin{equation*}
	\Xi(x,t) = \left\{
\begin{array}{l}
	\{1\} \quad\text{for}\quad  x = 0, \\
	\text{closed interval between}\ \xi(t)\ \text{and}\ 1\ \text{for}\ x \neq 0.
\end{array}
\right.
\end{equation*}
The convergence is in the sense of the Hausdorff distance of graphs over $[-L,L]\times[0,T]$ for any $T>0$.
\end{thm}

We also handle initial data not necessarily well-prepared.
 In this case, the equation \eqref{ELE} is altered because there are a few lower-order terms.
 If $V(y,0)=1-ce^{-\mu|x|}$ with $\mu\neq a$, we still get an explicit form corresponding to \eqref{ELE}.
 In these cases, the solution $\xi$ is explicitly represented by using the error function.
 We shall discuss these extensions for non-well-prepared data and the proof of Lemma \ref{LLE} in Section \ref{S2}.
 We also calculate numerically how the solution of the $\varepsilon$-problem converges by comparing it with the explicit solution of \eqref{ELE} and more general equations \eqref{EGFD}. 

We now come back to the singular limit problem of the Kobayashi--Warren--Carter system \eqref{EG1}--\eqref{EG2} with $\tau=\tau_1/\varepsilon$, $F(v)=a^2(v-1)^2$, $\alpha_0(v)=v^2$ as $\varepsilon\to0$.
 The equation \eqref{EG2} is of the total variation flow type, and its well-posedness for given $v$ is known when $v$ is independent of time and $\inf\alpha_w(v)>0$;
 in this case, \eqref{EG2} is the gradient flow of the weighted total variation $\int\alpha_0(v)|\nabla u|$ under $\alpha_w(v)$-weighted $L^2$ inner product if we impose the natural boundary condition like Neumann boundary condition.
 As in \cite{GGK}, we consider \eqref{EG2} for piecewise constant functions.
 For an interval $I=(p,q)$ and its given division $p=p_0<p_1<\cdots<p_m=q$, we consider a piecewise constant function of form
\[
	u(x,t) = h_j(t), \quad p_{j-1} < x \leq p_j \quad 
	(j = 1, \ldots, m).
\]
We interpret a solution by mimicking the notion of a solution when $v$ is independent of time and $\inf\alpha_w(v)>0$ (under the periodic boundary condition for simplicity).
 It is of the form
\begin{equation} \label{ECH}
\left \{
\begin{array}{l}
	\alpha_w(v)\partial_t u = \operatorname{div}(v^2 z),\quad |z| \leq 1, \\
	z(p_j) = \chi_j, \quad j = 0, 1, \ldots, m-1,
\end{array}
\right.
\end{equation}
where $\chi_i=\operatorname{sgn}(h_{j+1}-h_j)$;
 we identify $p_m=p_0$ by periodicity.
 See Figure \ref{Fpi}.
\begin{figure}[htb]
\centering
\includegraphics[width=6.5cm]{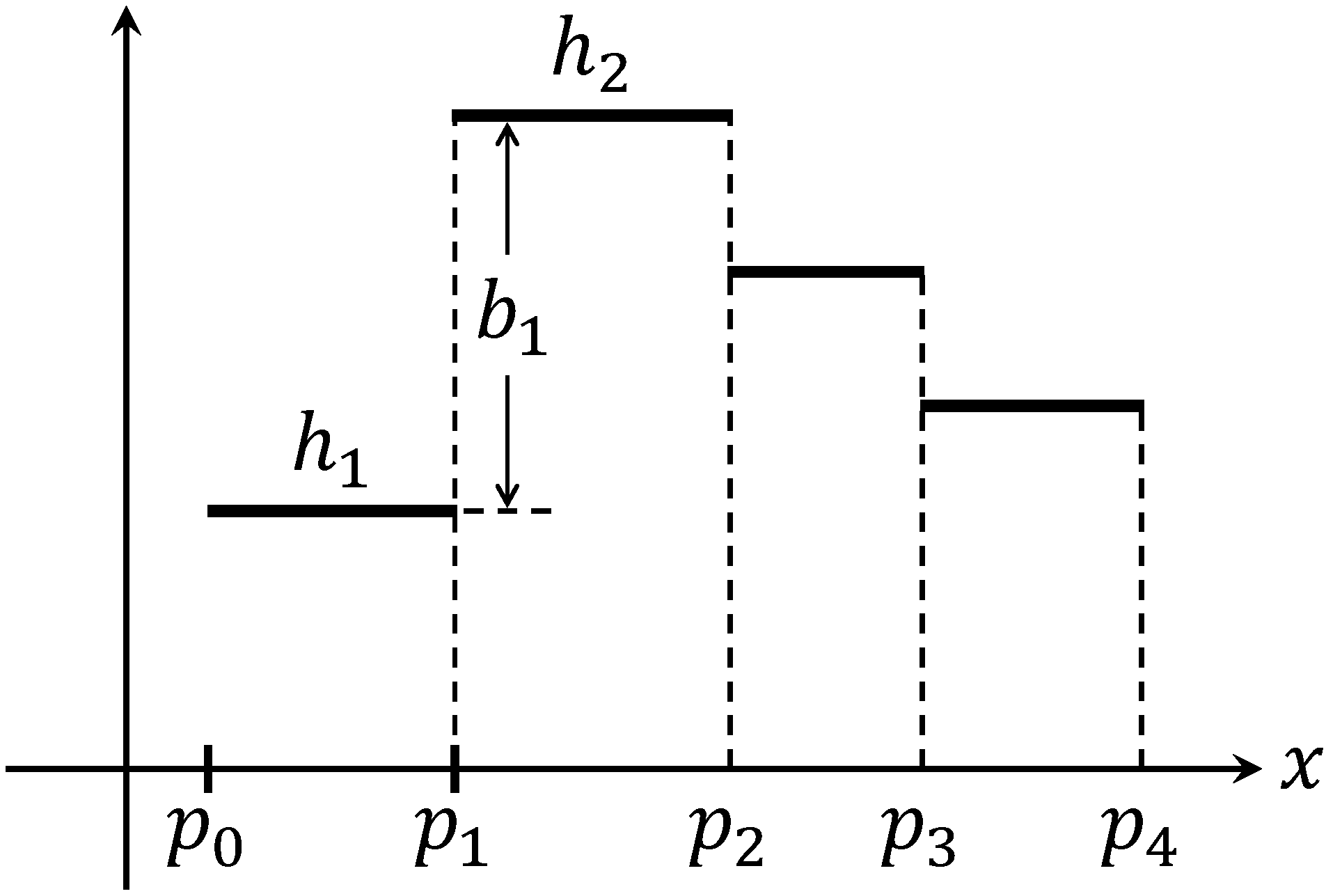}
\caption{graph of $u$ \\ $\chi_1=1$, $\chi_2=-1$} \label{Fpi}
\end{figure}
The function $z$ is called a Cahn-Hoffman vector field.

For a given $v$, it is unclear whether the solution stays in a class of piecewise constant functions \cite{GGK}, known as a facet-splitting problem.
 If $v$ is constant, it is well known that the solution stays in a class of piecewise constant functions.

We postulate that our system \eqref{ECH} has a (spatially) piecewise constant solution.
 Integrating the first equation of \eqref{ECH} in $(p_{j-1},p_j)$ yields
\[
	\int_{p_{j-1}}^{p_j} \alpha_w \left(v(x,t)\right)\, dx \frac{d}{dt} h_j(t)
	= v^2(p_j, t) \chi_j - v^2(p_{j-1},t) \chi_{j-1}.
\]
If $v_\varepsilon\to\Xi$ in the graph topology as $\varepsilon\to0$, we arrive at
\begin{equation} \label{ELT}
	\alpha_w(1) \frac{d}{dt} h_j
	= \frac{\xi_j^2 \chi_j - \xi_{j-1}^2 \chi_{j-1}}{p_j - p_{j-1}}.
\end{equation}
Since jump $b_j=|h_{j+1}-h_j|$ determines the $\xi_j$-equation as indicated in Theorem \ref{Tmain}, the equation for $\xi_j$ (assuming $\tau_1=1$) is expected to be
\begin{equation} \label{ELJ}
	M_a \partial_t \xi_j = -\operatorname{grad} E_\mathrm{sMM}^{0,b_j}(\xi_j)
	\left( = -2 \left( \left(|h_{j+1}-h_j| + a \right)\xi_j - a\right) \right),
\end{equation}
where $M_a f=\int_0^t m_a(t-s)f(s)\,ds$ for well-prepared initial data.
 Thus, the singular limit equation of the Kobayashi--Warren--Carter equation \eqref{EG1}--\eqref{EG2} as $\varepsilon\to0$ (under the periodic boundary condition) is expected to be the system  \eqref{ELT}--\eqref{ELJ} for $h_j$ and $\xi_j$.
 If we consider other boundary conditions, we always impose the homogeneous Neumann boundary condition for $v$, like \eqref{ENe}.
 Considering the Dirichlet boundary condition for $u$, $h_1$ and $h_m$ are fixed time-independent constants.
 We can impose the Neumann condition for $u$; in this case, imposing $z(p_0)=z(p_m)=0$ in \eqref{ECH} is natural.
 It is rather standard \cite{Po} to construct a unique local-in-time solution for a system of a fractional differential equation and an ordinary differential equation.

We note that the solvability of the initial value problem for the original Kobayashi--Warren--Carter system \eqref{EKWC1}--\eqref{EKWC2} is still an open problem, even in a one-dimensional setting.
 If we write it in the form of \eqref{EG1}--\eqref{EG2}, the difficulty stems from the fact that the weights $\alpha_w$ and $\alpha_0$ can vanish somewhere.
 In the literature, $\alpha_w$ is assumed to be away from zero.
 If $\alpha_0$ is allowed to vanish, $\Delta u$ term is added in the right-hand side of \eqref{EKWC2}.
 For example, instead of considering \eqref{EKWC2}, we consider
\begin{equation} \label{EMOKWC}
	\tau_0(v^2+\delta)u_t = s\operatorname{div} \left( (v^2+\delta')\nabla u / |\nabla u| + \mu\nabla u \right)
\end{equation}
with $\delta>0$, $\delta'\geq0$ and $\mu\geq0$ such that $\delta'+\mu>0$.
 The existence of a solution to \eqref{EKWC1} and \eqref{EMOKWC} with its large time behaviour is established in \cite{IKY, MoSh, MoShW1, SWat, SWY, WSh} under several homogeneous boundary conditions. 
 Unfortunately, the uniqueness of their solution is only known in a one-dimensional setting under the relaxation
term $\mu>0$ \cite[Theorem 2.2]{IKY}.
 The extension of these results to inhomogeneous boundary condition is not difficult.
 In \cite{MoShW2}, under non-homogeneous Dirichlet boundary conditions, structured patterns of stationary (i.e., time-independent) solutions were studied.
 In a one-dimensional setting, they thoroughly characterised all stationary solutions. 
 In this paper, we do not discuss convergence problems as $\varepsilon\to0$.

This paper is organised as follows.
 In Section \ref{S2}, we derive the limit equation both for well-prepared and non-well-prepared initial data.
 Most of the calculations there are very explicit. 
 In Section \ref{S3}, we prove Lemma \ref{LCon}.
 Section \ref{S4} gives a rigorous definition of the Dirichlet problem for \eqref{EG2}, assuming that $v$ is given and $\alpha_w(v)\equiv1$.
 In Section \ref{S5}, we give several numerical tests.

\section{Derivation of equations with fractional time derivative} \label{S2} 

We begin with recalling several elemental properties of the Laplace transform
\[
	\mathcal{L}[g](\lambda) := \int_0^\infty e^{-\lambda s} g(s)\, ds, \quad \lambda > 0
\]
for a locally integrable function $g$ in $[0,\infty)$.
 By definition,
\begin{equation} \label{ELS}
	\mathcal{L} \left[ e^{-\mu s} g(s) \right](\lambda)
	= \mathcal{L}[g](\lambda + \mu),
\end{equation}
and by definition of the Gamma function, we see
\[
	\mathcal{L} [ s^{\beta-1} ](\lambda)
	= \Gamma(\beta) \lambda^{-\beta}.
\]
We now arrive at a well-known formula
\begin{equation} \label{EL1}
	\mathcal{L} [ f_\beta^\alpha ](\lambda)
	= \mathcal{L} [ e^{-\alpha s} f_\beta^0 ]
	= (\lambda + \alpha)^{-\beta}
\end{equation}
for $\alpha,\beta>0$.
\begin{lemma} \label{LLap}
For $a>0$, let $m_a(t)$ be
\[
	m_a(t) = 2 \left\{ f_{1/2}^{a^2} (t) + a^2 \int_0^t f_{1/2}^{a^2} (s)\, ds - a \right\}.
\]
Then
\[
	\mathcal{L}[m_a](\lambda) = 2 \left(\frac{\sqrt{\lambda+a^2}-a}{\lambda}\right).
\]
Moreover, $m_a'(t)<0$ and $m_a(t)>0$ for $t>0$ with $\lim_{t\to\infty}m_a(t)=0$.
\end{lemma}
\begin{proof}
\quad
We recall that the Laplace transform of the convolution $(g_1*g_2)(t)=\int_0^t g_1(t-s)g_2(s)\,ds$ is given by
\begin{equation} \label{ECon}
	\mathcal{L} [g_1 * g_2]	= \mathcal{L} [g_1] \mathcal{L} [g_2].
\end{equation}
If we take $g_1\equiv1$, we see
\begin{equation} \label{EInt}
	\mathcal{L} \left[\int_0^t g_2(s)\,ds\right](\lambda)
	= \mathcal{L} [g_2] (\lambda) \lambda^{-1}
\end{equation}
since $\mathcal{L} [g_1](\lambda)=\lambda^{-1}$ by \eqref{EL1}.
 Setting $\alpha=a^2$, we apply \eqref{EL1} and \eqref{EInt} to get
\begin{align*}
	\mathcal{L} \left[\frac{m_a}{2}\right]
	&= (\lambda+\alpha)^{-1/2} + \alpha(\lambda+\alpha)^{-1/2} \lambda^{-1} - a\lambda^{-1} \\
	&= \left((\lambda+\alpha)^{1/2} - a\right)\lambda^{-1},
\end{align*}
which yields the desired formula for $\mathcal{L}[m_a]$.

We differentiate $m_a/2$ to get
\begin{align*}
	\frac12 m_a'(t)
	&= (-\alpha) f_{1/2}^\alpha(t) - \frac{1}{2t} f_{1/2}^\alpha(t) + \alpha f_{1/2}^\alpha(t) \\
	&= -\frac{1}{2t} f_{1/2}^\alpha(t)<0 \quad\text{for}\quad t>0.
\end{align*}
We observe that
\[
	\lim_{t\to\infty} \int_0^t f_{1/2}^\alpha(s)\, ds
	= \frac{1}{\Gamma\left(1/2\right)}\int_0^\infty e^{-\alpha s} s^{\frac12-1}\, ds = \alpha^{-1/2}.
\]
Thus, $\lim_{t\to\infty}m_a(t)=0$ since $\lim_{t\to\infty}f_{1/2}^\alpha(t)=0$, which this implies that $m_a(t)>0$ for $t>0$ since $m_a'(t)<0$.
\end{proof}
\begin{proof}[Proof of Lemma \ref{LLE}]
\quad
Since the property of $m_a$ has been proved in Lemma \ref{LLap}, it suffices to derive \eqref{ELE}.
 Studying the equation for $w=V-1$ instead of \eqref{ERR} is more convenient.
 The equation \eqref{ERR} for $w$ (with $\tau_1=1$) becomes
\begin{equation} \label{Ew}
	w_t - w_{xx} + a^2 w + 2b(w+1)\delta = 0, \quad
	x \in \mathbf{R}, \quad t > 0.
\end{equation}
The initial data $w(x,0)$ equals
\[
	w(x,0) = w_0(x)
\]
with $w_0(x)=-ce^{-a|x|}$.
 Let $\widehat{w}(x,\lambda)$ be the Laplace transform of $w$ in the $t$ variable, i.e.,
\[
	\widehat{w}(x,\lambda) = \mathcal{L}[w](x,\lambda)
	= \int_0^\infty e^{-\lambda t} w(x,t)\, dt.
\]
We note that
\[
	\widehat{w_t} = \lambda \widehat{w} - w_0, \quad
	\left. w \right|_{t=0} = w_0.
\]
Taking the Laplace transform of \eqref{Ew}, we arrive at
\begin{equation} \label{ELw}
	\lambda\widehat{w} - w_0 - \widehat{w}_{xx} + a^2\widehat{w} + 2b\left(\widehat{w}+\lambda^{-1}\right)\delta = 0,
\end{equation}
a second-order linear ordinary differential equation with a jump of the derivatives $\widehat{w}_x$.
 Since the coefficients are constants, we can solve \eqref{ELw} explicitly.

A bounded solution satisfying
\[
	\lambda\widehat{w} - w_0 - \widehat{w}_{xx} + a^2\widehat{w} = 0
	\quad\text{for}\quad x > 0 \quad\text{and}\quad x < 0
\]
is of form
\[
	\widehat{w} = Ae^{-\sqrt{\lambda+a^2}|x|} - \frac{c}{\lambda} e^{-a|x|}
\]
with $A\in\mathbf{R}$.
 One should determine $A$ so that \eqref{ELw} holds.
 We observe 
\[
	\widehat{w}_x = \left( -A\sqrt{\lambda+a^2} e^{-\sqrt{\lambda+a^2}|x|} + \frac{ca}{\lambda} e^{-a|x|} \right) (\operatorname{sgn}x)
\]
so that
\[
	\widehat{w}_{xx} = \left( -2A\sqrt{\lambda+a^2} + \frac{2ca}{\lambda} \right)\delta + A(\lambda+a^2) e^{-\sqrt{\lambda+a^2}|x|} - \frac{ca^2}{\lambda} e^{-a|x|}
\]
since $(\operatorname{sgn}x)'=2\delta$.
 The equation \eqref{ELw} imposes that
\[
	\text{the}\ \delta\ \text{part of}\ \widehat{w}_{xx} = 2b(\widehat{w} + \lambda^{-1})\delta,
\]
which implies
\[
	-2A\sqrt{\lambda+a^2} + \frac{2ca}{\lambda} = 2b\left(\widehat{w}(0,\lambda) + \lambda^{-1} \right).
\]
Since $\widehat{w}(0,\lambda)=A-c\lambda^{-1}$, we end up with
\[
	\left( \sqrt{\lambda+a^2}+b \right)A = - \frac{b(1-c)}{\lambda} + \frac{ca}{\lambda}.
\]
We set $\eta(t)=w(0,t)$ and observe that
\[
	\widehat{\eta} = A-\frac{c}{\lambda} = \frac{1}{\sqrt{\lambda+a^2}+b} \frac{1}{\lambda}(ca+bc-b) - \frac{c}{\lambda}.
\]

We next extract the time derivative of $\eta$.
 We note that
\[
	\widehat{\eta_t} = \lambda\widehat{\eta} - \eta(0)
\]
so that $\widehat{\eta_t}=\lambda\widehat{\eta}+c$ since $\eta(0)=-c$.
 Then
\[
	\widehat{\eta_t} = \frac{1}{\sqrt{\lambda+a^2}+b}(ca+bc-b).
\]
Multiplying $\sqrt{\lambda+a^2}+b$ and subtracting $a\widehat{\eta_t}$ from both sides, we get
\begin{align*}
	\left( \sqrt{\lambda+a^2}-a \right)\widehat{\eta_t}
	&= -(b+a) \widehat{\eta_t} + (ca+bc-b) \\
	&= -(b+a) \lambda\widehat{\eta} - (b+a)c + (ca+bc-b) \\
	&= \lambda\left( -(b+a)\widehat{\eta}-\frac{b}{\lambda} \right).
\end{align*}
The energy for $\xi$ equals
\[
	E_\mathrm{sMM}^{0,b}(\xi) = b\xi^2 + a(\xi-1)^2.
\]
Since $\eta=\xi-1$, we see $\operatorname{grad}E_\mathrm{sMM}^{0,b}(\xi)=2(b+a)\eta+2b$.
 Thus
\begin{equation} \label{EEN}
	2\left( \frac{\sqrt{\lambda+a^2}-a}{\lambda}\right)\widehat{\eta_t} = -2(b+a)\widehat{\eta} - 2b\widehat{1}.
\end{equation}
By Lemma \ref{LLap} and \eqref{ECon}, we now conclude that $\xi=\eta+1$ solves \eqref{ELE} by taking the inverse Laplace transform of \eqref{EEN}.
\end{proof}
In Lemma \ref{LLE}, we only considered well-prepared initial data, meaning that $V(y,t)$ is a stationary solution of \eqref{ERR} in $\mathbf{R}\backslash\{0\}$.
 We take this opportunity to write a general equation corresponding to \eqref{ELE} starting from general initial data.
 We set
\[
	(M_a f)(t) = m_a * f(t) = \int_0^t m_a(t-s) f(s)\, ds,
\]
and the equation \eqref{ELE} is
\begin{equation} \label{ELE2}
	M_a \xi_s = -\operatorname{grad}E_\mathrm{sMM}^{0,b} (\xi).
\end{equation}
For general initial data, our equation corresponding to \eqref{ELE2} becomes more complicated than \eqref{ELE2}.

We set
\[
	G_\lambda^a(y) = \frac{e^{-\sqrt{\lambda+a^2}|y|}}{2\sqrt{\lambda+a^2}},
\]
which is the Green function of $-\partial_y^2+\sigma$ with $\sigma=\lambda+a^2$, i.e.,
\[
	-\partial_y^2 G_\lambda^a + (\lambda+a^2) G_\lambda^a = \delta.
\]
For $w_0\in L^\infty(\mathbf{R})$, we set
\[
	g^a(\lambda,w_0) = (G_\lambda^a *_x w_0)(0),
\]
where $*_x$ denotes the convolution in space, i.e.,
\[
	(G_\lambda^a *_x w_0)(y) = \int_{-\infty}^\infty G_\lambda^a(y-z) w_0(z)\,dz.
\]
\begin{lemma} \label{LGeq}
Let $\tau_1=1$.
 Let $V$ be the bounded solution of \eqref{ERR} in $\mathbf{R}\times(0,\infty)$ with initial data $V(y,0)=1+w_0$, where $w_0$ is bounded and Lipschitz continuous in $\mathbf{R}$.
 Then $\xi(t)=V(0,t)$ solves
\begin{equation} \label{ELEG}
	M_a \xi_t + m_a\left(\xi(0)-1\right) - \mathcal{L}^{-1} \left[2\sqrt{\lambda+a^2} g^a(\lambda,w_0)\right]
	= - \operatorname{grad} E_\mathrm{sMM}^{0,b}(\xi).
\end{equation}
\end{lemma}
\begin{proof}
\quad
If $w_0$ is Lipschitz, then we easily see that $|w_t|t^{1/2}$ is bounded (near $t=0$) so that the integrand of
\begin{equation*}
	\mathcal{L}[w_t](\lambda)=\int_0^\infty e^{-\lambda s}w_s(x,s)\,ds
\end{equation*}
is integrable for all $x\in\mathbf{R}$ near $s=0$; see Lemma~\ref{Lemma3.1}.

We argue in the same way to prove Lemma \ref{LLE}.
 We begin with \eqref{ELw}, where $w=V-1$.
 Its solution outside $x=0$ is given as
\[
	\widehat{w} = Ae^{-\sqrt{\lambda+a^2}|x|} + g^a(\lambda,w_0)
\]
with $A\in\mathbf{R}$.
 As before, we determine $A$ and obtain that $\eta=\xi-1$ satisfies
\begin{equation} \label{ELSOL}
	\widehat{\eta}(\lambda) = \frac{-b}{\sqrt{\lambda+a^2}+b} \left( g^a(\lambda,w_0) + \frac1\lambda \right)
	+ g^a(\lambda,w_0).
\end{equation}

Since $\widehat{\eta_t}=\lambda\widehat{\eta}-\eta(0)$, we proceed with
\begin{align*}
	&\left(\sqrt{\lambda+a^2}-a\right)\widehat{\eta_t} \\
	&= -(b+a)\widehat{\eta_t} -\lambda b\left(g^a(\lambda,w_0)+\frac1\lambda \right)
	+ \left(\sqrt{\lambda+a^2}+b\right) \left(\lambda g^a(\lambda,w_0)-\eta(0) \right) \\
	&= -(b+a)\lambda\widehat{\eta} + (b+a)\eta(0) -\lambda b\left(g^a(\lambda,w_0)+\frac1\lambda \right)\\
	&\hspace{150pt}+ \left(\sqrt{\lambda+a^2}+b\right) \left(\lambda g^a(\lambda,w_0)-\eta(0) \right) \\
	&= -(b+a)\lambda\widehat{\eta} -\lambda\frac{b}{\lambda}  + \left(a-\sqrt{\lambda+a^2}\right) \eta(0) +\sqrt{\lambda+a^2} \lambda g^a(\lambda,w_0).
\end{align*}
Thus
\begin{multline*}
	\frac{2\left(\sqrt{\lambda+a^2}-a\right)}{\lambda} \widehat{\eta_t}
	= -2(b+a)\widehat{\eta} \\
	- 2\frac{b}{\lambda} - \frac{2\left(\sqrt{\lambda+a^2}-a\right)}{\lambda}\eta(0)
	+ 2\sqrt{\lambda+a^2} g^a(\lambda,w_0).
\end{multline*}
Taking the inverse Laplace transform, we obtain
\[
	M_a \eta_t = -\operatorname{grad}E_\mathrm{sMM}^{0,b}(\eta+1) - m_a(t)\eta(0)
	+ \mathcal{L}^{-1} \left(2\sqrt{\lambda+a^2} g^a(\lambda,w_0) \right),
\]
the same as \eqref{ELEG}.
\end{proof}
The term $\mathcal{L}^{-1} \left[2\sqrt{\lambda+a^2} g^a\right]$ has a more explicit form.
 Let $E(x,t)$ be the Gauss kernel, i.e.,
\[
	E(x,t) = \frac{1}{\sqrt{4\pi t}} e^{-|x|^2/4t}.
\]
We know its Laplace transform (as a function of $t$) is
\[
	\mathcal{L}[E](\lambda) = \frac{1}{2\lambda^{1/2}} e^{-|x|\lambda^{1/2}}.
\]
Suppose we set $E^a=e^{-a^2 t}E$, then
\[
	\mathcal{L}[E^a](\lambda) = \frac{1}{2\sqrt{\lambda+a^2}} e^{-\sqrt{\lambda+a^2}|x|} = G_\lambda^a(x).
\]
Since $\mathcal{L}\left[f_{1/2}^{a^2}\right]=(\lambda+a^2)^{-1/2}$ and $\mathcal{L}[g_1*g_2] = \mathcal{L}[g_1] \mathcal{L}[g_2]$, we end up with
\[
	\mathcal{L} \left[ f_{1/2}^{a^2} * E^a \right](\lambda)
	= \frac{1}{2(\lambda+a^2)} e^{-\sqrt{\lambda+a^2}|x|}.
\]
Thus
\[
	\mathcal{L} \left[ \partial_t \left(f_{1/2}^{a^2} * E^a \right) + a^2 \left(f_{1/2}^{a^2} * E^a \right) \right]
	= \frac12 e^{-\sqrt{\lambda+a^2}|x|}.
\]
Since $\sqrt{\lambda+a^2}g^a=\frac12\left(e^{-\sqrt{\lambda+a^2}|x|} *_x w_0 \right)(0,t)$, we conclude that
\[
	\mathcal{L}^{-1} \left[ 2 \sqrt{\lambda+a^2} g^a \right] = 2\partial_t \left( f_{1/2}^{a^2} * (E^a *_x w_0)(0) \right) + 2 f_{1/2}^{a^2} * (E^a *_x w_0)(0).
\]

If $w_0=-ce^{-\mu|x|}$, we observe that
\begin{equation} \label{ELGA}
	\mathcal{L}^{-1} \left[ 2 \sqrt{\lambda+a^2} g^a \right] = -2c \left( f_{1/2}^{a^2} - \mu e^{(\mu^2-a^2)t} \operatorname{erfc}\left(\mu\sqrt{t}\right) \right),
\end{equation}
where $\operatorname{erfc}$ denotes the error function, i.e.,
\[
	\operatorname{erfc}(s) = \frac{2}{\sqrt{\pi}} \int_s^\infty e^{-\tau^2}\, d\tau.
\]
Indeed, we proceed with
\begin{align}
\begin{aligned}\label{ELGa}
 	g^a(\lambda) &= G_\lambda^a *_x \left(-ce^{-\mu|x|}\right)(0) 
	= -c\frac{1}{2\sqrt{\lambda+a^2}}
	\int_{-\infty}^\infty e^{-\sqrt{\lambda+a^2}|0-y|} e^{-\mu|y|}\, dy \\
	&= -\frac{c}{\sqrt{\lambda+a^2}}
	\int_0^\infty e^{-\left(\mu+\sqrt{\lambda+a^2}\right)y} \,dy
	= - \frac{c}{\sqrt{\lambda+a^2}\left(\mu+\sqrt{\lambda+a^2}\right)}.
\end{aligned}
\end{align}
By a direct calculation, we observe that
\begin{align*}
 	&\mathcal{L} \left[ \operatorname{erfc}\left(\mu\sqrt{t}\right) \right]
	= \int_0^\infty e^{-\lambda t} \left(\frac{2}{\sqrt{\pi}} \int_{\mu\sqrt{t}}^\infty e^{-s^2}\, ds \right)dt \\
	&\quad= \frac{2}{\sqrt{\pi}} \int_0^\infty \left[ \int_0^{s^2/\mu^2} e^{-s^2} e^{-\lambda t}\,dt \right]ds
	= \frac{2}{\sqrt{\pi}} \int_0^\infty e^{-s^2} \left[ \int_0^{s^2/\mu^2} e^{-\lambda t}\,dt \right]ds \\
	&\quad= \frac{2}{\sqrt{\pi}} \int_0^\infty e^{-s^2} \left( \frac{1-e^{\lambda s^2/\mu^2}}{\lambda} \right)ds
	= \frac{2}{\sqrt{\pi}} \frac1\lambda \left( \frac{\sqrt{\pi}}{2} - \frac{\sqrt{\pi}}{2} \sqrt{\frac{\mu^2}{\mu^2+\lambda}} \right) \\
	&\quad= \frac1\lambda \left( 1 - \frac{\mu}{\sqrt{\lambda+\mu^2}} \right).
\end{align*}
Thus,
\begin{align}
\begin{aligned} \label{ELERR}
 	\mathcal{L} \left[ f_{1/2}^0 - \mu e^{\mu^2 t} \operatorname{erfc}\left(\mu\sqrt{t}\right) \right]
	&= \frac{1}{\sqrt{\lambda}} - \frac{\mu}{\lambda-\mu^2} \left( 1 - \frac{\mu}{\sqrt{\lambda}} \right)
	= \frac{\lambda-\mu^2-\sqrt{\lambda}\mu+\mu^2}{\sqrt{\lambda}(\lambda-\mu^2)} \\
	&= \frac{\sqrt{\lambda}-\mu}{\lambda-\mu^2} = \frac{1}{\sqrt{\lambda}+\mu},
\end{aligned}
\end{align}
where we invoked \eqref{EL1} and \eqref{ELS}.
 We set $q^\mu(t)=f_{1/2}^0(t)-\mu e^{\mu^2 t}\operatorname{erfc}\left(\mu\sqrt{t}\right)$ so that \eqref{ELERR} becomes
\begin{equation} \label{ELERR2}
	\mathcal{L}[q^\mu] = \frac{1}{\sqrt{\lambda}+\mu}.
\end{equation} 
Using \eqref{ELS} again, we now obtain \eqref{ELGA}.

The formula \eqref{ELERR} gives another representation of $m_a$.
 Indeed, since
\[
	\frac12 \mathcal{L}[m_a] (\lambda) = \frac{\sqrt{\lambda+a^2}-a}{\lambda}
	= \frac{1}{\sqrt{\lambda+a^2}+a},
\]
we see
\begin{equation} \label{EMERF}
	\frac12 m_a(t) = e^{-a^2t} q^a(t) = f_{1/2}^{a^2}(t) - a \operatorname{erfc}\left(a\sqrt{t}\right),
\end{equation}
in particular, implies that $m_a(t)\leq 2f_{1/2}^{a^2}(t)$ and
\[
	\operatorname{erfc}\left(a\sqrt{t}\right)
	= 1 - a\int_0^t f_{1/2}^{a^2}(s)\,ds. 
\]
Moreover, by \eqref{EMERF}, we see
\[
	q^a(t) = e^{a^2t} m_a(t)/2.
\]
 Therefore, the positivity of $m_a$ (Lemma \ref{LLap}) implies $q^a(t)>0$ for $t>0$.
 The formula \eqref{ELEG} in Lemma \ref{LGeq} becomes
\begin{equation} \label{EGFD}
	M_a \xi_t + 2cae^{-a^2t} \operatorname{erfc}\left(a\sqrt{t}\right)
	- 2c\mu e^{(\mu^2-a^2)t} \left(\operatorname{erfc}\left(\mu\sqrt{t}\right)\right)
	= -\operatorname{grad} E_\mathrm{sMM}^{0,b}(\xi) 
\end{equation}
if $w_0=-ce^{-\mu|x|}$.
 If $\mu=a$, this is reduced to \eqref{ELE} (or \eqref{ELE2}).

One can give an explicit form of a solution of \eqref{EGFD} starting from $\xi(0)=1-c$.
 We plug \eqref{ELGa} into \eqref{ELSOL} to get
\begin{align}
\begin{aligned}\label{ELSOLF}
 	\widehat{\eta}(\lambda) &= \frac{\sqrt{\lambda+a^2}}{\sqrt{\lambda+a^2}+b} g^a(\lambda) + \frac{-b}{\lambda\left(\sqrt{\lambda+a^2}+b\right)} \\
 	&= \frac{1}{\sqrt{\lambda+a^2}+b} \left(\frac{-c}{\mu+\sqrt{\lambda+a^2}} - \frac{b}{\lambda} \right).
\end{aligned}
\end{align}
Since
\[
	\mathcal{L}\left[e^{-a^2t}q^\mu\right] = \frac{1}{\sqrt{\lambda+a^2}+\mu}
\]
by \eqref{ELERR} and \eqref{ELS}, we have
\begin{align*} 
 	\eta &= -ce^{-a^2t} q^b * e^{-a^2t} q^\mu
 	- b \int_0^t e^{-a^2s} q^b (s)\, ds \\
 	&= -ce^{-a^2t} (q^b * q^\mu)
 	- b \int_0^t e^{-a^2s} q^b (s)\, ds.
\end{align*}
However, the calculation of $q^b*q^\mu$ is quite involved, and it is easier to calculate $\widehat{\eta}$ in \eqref{ELSOLF} more.
 We proceed
\begin{align*} 
 	\widehat{\eta}(\lambda) 
	&= \frac{1}{\sqrt{\lambda+a^2}+b} \left( \frac{c\left(\sqrt{\lambda+a^2}-\mu\right)}{\lambda+a^2-\mu^2} - \frac{b}{\lambda} \right)\\
 	&= \frac{1}{\sqrt{\lambda+a^2}+b} \left( \frac{-c\left(\left(\sqrt{\lambda+a^2}+b\right)-b-\mu\right)}{\lambda+(a^2-\mu^2)} - \frac{b}{\lambda} \right) \\
	&= \frac{1}{\sqrt{\lambda+a^2}+b} \frac{(b+\mu)c}{\lambda+(a^2-\mu^2)}
	- \frac{c}{\lambda+(a^2-\mu^2)} - \frac{1}{\sqrt{\lambda+a^2}+b} \frac{b}{\lambda}\\
	&\equiv I + I\hspace{-0.2em}I + I\hspace{-0.2em}I\hspace{-0.2em}I.
\end{align*}
It is easy to see that
\[
	\mathcal{L}^{-1}(I\hspace{-0.3em}I) = -ce^{-(a^2-\mu^2)t}.
\]
As we already observed,
\[
	\mathcal{L}^{-1}(I\hspace{-0.3em}I\hspace{-0.3em}I) = -b \int_0^t e^{-a^2s} q^b(s)\,ds.
\]
For the first term,
\[
	\mathcal{L}^{-1}(I) = -c(b+\mu) e^{-(a^2-\mu^2)t} * e^{-a^2t} q^b.
\]
By definition,
\begin{align*} 
 	e^{-(a^2-\mu^2)t} * e^{-c^2t} q^b
 	&= \int_0^t e^{-(a^2-\mu^2)(t-s)} e^{-a^2s} q^b (s)\, ds \\
 	&= e^{-(a^2-\mu^2)t} \int_0^t e^{-\mu^2s} q^b (s)\, ds.
\end{align*}
We thus conclude that
\begin{equation} \label{EEXPS}
	\eta(t) = c \left[(\mu+b) \int_0^t e^{-\mu^2s} q^b (s)\, ds -1 \right]
	e^{-(a^2-\mu^2)t} -b \int_0^t e^{-a^2s} q^b (s)\, ds.
\end{equation}
Thus, $\xi=\eta+1$ is the solution of \eqref{EGFD} with $\xi(0)=1-c$.

From this solution formula, we can establish the solution's large-time behaviour.
 We set
\[
	\eta = \bar{\eta}(t) + \eta_e(t), \quad
	\bar{\eta}(t) = -b \int_0^t e^{-a^2s} q^b (s)\, ds.
\]
\begin{lemma} \label{LLT}
\begin{enumerate}
\item[(i)] The function $\bar{\eta}$ is positive, monotonically decreasing, and converging to $-b/(b+c)$ as $t\to\infty$. 
\item[(i\hspace{-0.1em}i)] The estimate
\[
	\left| \frac{\eta_e}{c(\mu+b)} \right| 
	\leq e^{-a^2t} \int_t^\infty q^b(s)\,ds \leq e^{-a^2t} \frac1b
\]
holds for $t>0$.
 In particular,
\[
	\lim_{t\to\infty}\eta_e e^{a^2t} = 0.
\]
\end{enumerate}
\end{lemma}
\begin{proof}
\begin{enumerate}
\item[(i)] Since $q^b\geq0$, the monotonicity is clear.
 We observe that
\[
	\int_0^\infty e^{-a^2s}q^b(s)\,ds
	= \mathcal{L} \left[q^b\right] (0+a^2) = \frac{1}{\sqrt{a^2}+b}
\]
by \eqref{ELERR}.
 Thus $\lim_{t\to\infty} \bar{\eta}(t)=-b/(b+a)$. 
\item[(i\hspace{-0.1em}i)] Since
\[
	\int_0^\infty e^{-\mu^2s}q^b(s)\,ds
	= \frac{1}{\mu+b},
\]
we observe that
\begin{align*}
	\eta_e &= -c(\mu+b) \int_t^\infty e^{-\mu^2s} q^b(s)\,ds\,e^{-(a^2-\mu^2)t} \\
	&= -c(\mu+b) \int_t^\infty e^{\mu^2(t-s)} q^b(s)\,ds\,e^{-a^2t}.
\end{align*}
Since $e^{\mu^2(t-s)} \leq 1$ for $s\geq t$, this implies $\left| \eta_e/c(\mu+b)\right| \leq e^{-a^2t} \int_t^\infty q^b(s)\,ds \leq e^{-a^2t} \int_0^\infty q^b(s)\,ds = e^{-a^2t}/b$.
 The proof is now complete.
\end{enumerate}
\end{proof}
\begin{remark} \label{RCS}
We consider the initial-boundary value problem for
\begin{equation*}
    \left\{
    \begin{alignedat}{2}
    w_t - (-x)^\alpha w_{xx} &= 0 &&\text{in}\quad (-\infty,0) \times (0,\infty), \\
    w(0,t) &= \eta(t) &\quad&\text{for}\quad t>0, \quad \eta(0)=0, \\
    \lim_{x\to-\infty}w(x,t) &= 0 &&\text{for}\quad t>0 \\ 
    w(x,0) &= 0 &&\text{for}\quad x<0.
    \end{alignedat}
    \right.
\end{equation*}
Then, as in \cite{CS}, we obtain that
\[
	w_x(0,t) = c_\gamma \partial_t^\gamma \eta,
\]
where $\gamma=1/(2-\alpha)$ with some constant $c_\gamma>0$ provided that $\alpha<1$.
 Indeed as in \cite{CS}, let $\psi$ be a solution of
\begin{equation*}
    \left\{
    \begin{alignedat}{2}
    \psi - (-x)^\alpha \psi_{xx} &= 0 \quad\text{in}\quad (-\infty,0), \\
    \psi(0) &= 1, \\
    \lim_{x\to-\infty}\psi(x) &= 0.
    \end{alignedat}
    \right.
\end{equation*}
Since the Laplace transform $\widehat{w}$ of $w$ satisfies
\[
	\lambda\widehat{w} - (-x)^\alpha\widehat{w}_{xx} = 0, \quad
	\widehat{w}(0,\lambda) = \widehat{\eta}(\lambda)
\]
we see, by scaling, that
\[
	\widehat{w}(x,\lambda) = \widehat{\eta}(\lambda) \psi\left(\lambda^{1/(2-\alpha)}x \right).
\]
Thus
\begin{align*}
	\partial_x \widehat{w}(0,\lambda) &= \lambda^{1/(2-\alpha)} \psi'(0) \widehat{\eta}(\lambda) \\
	&= \psi'(0)\lambda^{\gamma-1} \widehat{\eta_t}(\lambda)
\end{align*}
since $f(0)=0$.
 Thus
\[
	w_x(0,t) = \psi'(0) f_{1-\gamma}^0 * \eta_t.
\]
If $\gamma<1$, then $f_{1-\gamma}^0$ is integrable.
 As noted in \cite{CS}, $\psi'(0)$ exists (even for the degenerate case, i.e., $\alpha>0$) and $\psi'(0)>0$.
 Thus,
\[
	w_x(0,t) = c_\gamma \partial_t^\gamma \eta
	\quad\text{with}\quad c_\gamma = \psi'(0)
\]
at least for $\gamma=1/(2-\alpha)<1$, i.e., $\alpha<1$. 
\end{remark}

\begin{remark} \label{RFRC}
The reader might be interested in how fractional partial differential equations like fractional diffusion equations are derived.
 We consider
\begin{equation*}
    \left\{
    \begin{alignedat}{2}
    w_t - (-x)^\alpha w_{xx} &= 0 &\quad\text{in}\quad &(-\infty,0) \times \mathbf{R}^{n-1} \times (0,\infty), \\
    w_x - \Delta_y w &= f &\quad\text{on}\quad &(-\infty,0) \times \mathbf{R}^{n-1} \times (0,\infty), \\
    w(x,y,0) &= 0 &\quad\text{on}\quad &(-\infty,0) \times \mathbf{R}^{n-1},
    \end{alignedat}
    \right.
\end{equation*}
where $f=f(y,t)$ is a given function.
 Then by Remark \ref{RCS}, the equation for $\eta(y,t)=w(0,y,t)$ is formally obtained as
\begin{equation} \label{EFDIF}
	c_\gamma\partial_t^\gamma \eta - \Delta_y \eta = f
	\quad\text{in}\quad \mathbf{R}^{n-1} \times (0,\infty).
\end{equation}
This type of equation is a kind of fractional diffusion that has been well-studied; see \cite{KRY,Z2}.
 Here, we briefly recall only the well-posedness of its initial-boundary value problem for \eqref{EFDIF} in a domain.
 In the framework of distributions, the well-posedness of its initial-boundary value problems has been established in \cite{SY,Z1} by using the Galerkin method.
 The unique existence of viscosity solutions for \eqref{EFDIF}, including general nonlinear problems, has been established in \cite{GN,Na} and also in \cite{TY} for the whole space $\mathbf{R}^{n-1}$.
 The scope of equations these theories apply is different. 
 However, it has been proved in \cite{GMS} that two notions of solutions (viscosity solution and distributional solution) agree for \eqref{EFDIF} when we consider the Dirichlet problem in a smooth bounded domain.
\end{remark}

\section{Convergence} \label{S3} 

We shall prove Lemma \ref{LCon}.
\begin{proof}[Proof of Lemma \ref{LCon}]
\quad
For a function $f$ on $(-L,L)$, we decompose it into its odd and even parts, i.e.,
\[
	f_\mathrm{odd}(x) := \frac{f(x)-f(-x)}{2}, \quad
	f_\mathrm{even}(x) := \frac{f(x)+f(-x)}{2},
\]
so that $f=f_\mathrm{odd}+f_\mathrm{even}$.
 By the structure of the equation, $V_\mathrm{odd}^\varepsilon$ and $V_\mathrm{even}^\varepsilon$ solve the equation \eqref{ERR} separately.

At first glance, the locally uniform convergence follows from the maximum or comparison principles for a linear parabolic equation \cite{PW}.
 However, a direct application of the maximum principle is impossible since the domains of functions $V^\varepsilon$ and $V$ are different.
 We first show the convergence where initial data is smooth.

For the odd part, the term $\partial_y\{1_{y>0}\}V$ does not affect since $V=0$ at $x=0$.
 Thus, the equation \eqref{ERR} is reduced to
\begin{equation}
	\left \{
\begin{array}{rll} \label{ENoS1}
    \tau_1 V_t &\hspace{-0.5em}= V_{yy} - a^2(V-1), & |y| < L/\varepsilon \\
    \left. V \right|_{t=0} &\hspace{-0.5em}= V_0^\varepsilon, & V_y(\pm L/\varepsilon,t) = 0 \quad\text{for}\quad t > 0,
\end{array}
\right.
\end{equation}
where $V_0^\varepsilon(y)=V^\varepsilon (y,0)$ for $|y|<L/\varepsilon$.
 Let $V^\varepsilon$ be its solution.

We extend an odd function $V_0^\varepsilon$ outside $\left(L/\varepsilon,3L/\varepsilon\right)$ for $x$ to be ``even'' with respect to $L/\varepsilon$, i.e.,
\[
	\widetilde{V_0^\varepsilon} \left( x-L/\varepsilon \right) = \widetilde{V_0^\varepsilon} \left( L/\varepsilon-x \right), \quad
	x \in \left( L/\varepsilon, 3L/\varepsilon \right),
\]
where $\widetilde{V_0^\varepsilon}$ is its extension.
 We extend outside $\left(-L/\varepsilon,3L/\varepsilon\right)$ so that the extension $\overline{V_0^\varepsilon}$ is periodic in $\mathbf{R}$ with period $4L/\varepsilon$.
 Since $V_0^\varepsilon$ is even with respect to $L/\varepsilon$, $\left(\overline{V_0^\varepsilon}\right)_y \left(\pm L/\varepsilon\right)=0$ if $V_{0y}^\varepsilon\left(\pm L/\varepsilon\right)=0$ and smooth.
 Solution $V^\varepsilon$ is the restriction on $\left(-L/\varepsilon,L/\varepsilon\right)$ of a solution $W$ of
\begin{equation} 
	\left \{
\begin{array}{rl} \label{ENoS2} 
    &\tau_1 W_t = W_{yy} - a^2(W-1), \quad y \in \mathbf{R} \\
    &W |_{t=0} = \overline{V_0^\varepsilon}.
\end{array}
\right.
\end{equation}
Although the maximum principle implies
\[
	\left\lVert W - V^\varepsilon \right\rVert_{L^\infty(\mathbf{R})}(t)
	\leq \left\lVert \overline{V_0^\varepsilon} - V_0 \right\rVert_{L^\infty(\mathbf{R})},
\]
our assumption of the convergence $V_0^\varepsilon\to V_0$ does not guarantee $\left\lVert \overline{V_0^\varepsilon} - V_0 \right\rVert_{L^\infty(\mathbf{R})}\to0$.
 We argue differently.

We approximate $V_0$ by $V_{0\delta}=V_0*\rho_\delta$, where $\rho_\delta$ is a symmetric mollifier.
 We also approximate $V_0^\varepsilon$ by
\[
	V_{0\delta}^\varepsilon = \overline{V_0^\varepsilon} * \rho_\delta.
\]
Let $V_\delta^\varepsilon$ be the restriction of $W_\delta^\varepsilon$ of \eqref{ENoS2} with initial data $V_{0\delta}^\varepsilon$.
 It follows from the parity and periodic condition that this $V_\delta^\varepsilon$ solves \eqref{ENoS1}; moreover, $W_\delta^\varepsilon=W^\varepsilon*\rho_\delta$, where $W^\varepsilon$ is the solution of \eqref{ENoS2} with initial data $\overline{V_0^\varepsilon}$.
 Let $V_\delta$ be the bounded solution of \eqref{ENoS2} with initial data $V_{0\delta}$ and again $V_\delta=V*\rho_\delta$.
 These properties follow from the fact that our equation is of constant coefficients.
 For fixed $\delta>0$, we observe that
\[
	V_\delta^\varepsilon \to V_\delta \quad\text{in}\quad
	L^\infty \left( (-M, M)\times(0, T) \right)
\]
for $M>0$.
Indeed, by the maximum principle
\begin{align*}
	\left\lVert V_\delta^\varepsilon - 1 \right\rVert_{\infty,\varepsilon}(t) &\leq \left\lVert V_{0\delta}^\varepsilon - 1 \right\rVert_{\infty,\varepsilon},	\\
	\left\lVert \partial_y V_\delta^\varepsilon \right\rVert_{\infty,\varepsilon}(t) &\leq \left\lVert \partial_y V_{0\delta}^\varepsilon \right\rVert_{\infty,\varepsilon},	\\
	\left\lVert \partial_t V_\delta^\varepsilon \right\rVert_{\infty,\varepsilon}(t) &\leq \left\lVert \partial_t V_{0\delta}^\varepsilon \right\rVert_{\infty,\varepsilon},	
\end{align*}
where we interpret $\tau_1\partial_tV_{0\delta}^\varepsilon=V_{0\delta yy}^\varepsilon-a^2\left(V_{0\delta}^\varepsilon-1\right)$ and $\|\cdot\|_{\infty,\varepsilon}$ is the sup norm on $\left(-L/\varepsilon, L/\varepsilon\right)$.
 Because of the mollifier, the right-hand side is bounded by a constant multiple of $\left\lVert V_0^\varepsilon\right\rVert_{\infty,\varepsilon}$, which is uniformly bounded in $\varepsilon<1$.
 By the Arzel\`a--Ascoli theorem, $V^\varepsilon$ converges (locally uniformly in $\mathbf{R}\times[0,\infty)$) to a bounded (weak) solution to \eqref{ENoS2} with initial data $V_0$ by taking a subsequence.
 Since $V$ is bounded, by the uniqueness of the limit problem, the convergence is now full (without taking a subsequence).
 Note that we only invoke the locally uniform convergence of $V_0^\varepsilon$ to $V_0$ other than uniform bound on derivatives.

We note that
\[
	V - V^\varepsilon = V- V_\delta + V_\delta - V_\delta^\varepsilon + V_\delta^\varepsilon - V^\varepsilon
\]
and observe that
\[
	\|V - V^\varepsilon\|(t) \leq \|V- V_\delta\|(t) + \|V_\delta - V_\delta^\varepsilon\|(t) + \|V_\delta^\varepsilon - V^\varepsilon\|(t)
	=: I + I\hspace{-0.3em}I + I\hspace{-0.3em}I\hspace{-0.3em}I
\]
where the norm is $\|\cdot\|$ taken in $L^\infty(-M,M)$ for $M>0$.
 By the maximum principle,
\begin{align*}
	&\|V_\delta^\varepsilon - V^\varepsilon\|(t) \leq \|V_{0\delta}^\varepsilon - V_0^\varepsilon\|_{L^\infty(-L/\varepsilon,L/\varepsilon)}
	&\hspace{-5em}\text{for}\quad L > \varepsilon M, \\
	&\|V - V_\delta\|(t) \leq \|V_{0\delta} - V_0\|_{L^\infty(\mathbf{R})}
	&\hspace{-5em}\text{for}\quad L > \varepsilon M.
\end{align*}
Since
\begin{align*}
	&V_{0\delta}^\varepsilon = \rho_\delta * \left(\overline{V_0^\varepsilon} - V_0\right) + \rho_\delta * V_0, \\
	&V_0^\varepsilon = \overline{V_0^\varepsilon} - V_0 + V_0,
\end{align*}
and $\|\rho_\delta*f\|_\infty\leq\|f\|_\infty$, we see that
\[
	\|V_{0\delta}^\varepsilon - V_0^\varepsilon\|_{L^\infty(-L/\varepsilon,L/\varepsilon)}
	\leq 2 \|\overline{V_0^\varepsilon} - V_0\|_{\infty,\varepsilon}
	+ \|\rho_\delta * V_0 - V_0\|_\infty.
\]
Thus
\[
	\sup_{0<t<T} I\hspace{-0.3em}I\hspace{-0.3em}I \leq 2 \|V_0^\varepsilon - V_0\|_{\infty,\varepsilon}
	+ \|\rho_\delta * V_0 - V_0\|_\infty.
\]
Fixing $\delta>0$ and sending $\varepsilon\to0$, we observe that
\[
	\varlimsup_{\varepsilon\downarrow0} \sup_{0<t<T} (I + I\hspace{-0.3em}I + I\hspace{-0.3em}I\hspace{-0.3em}I)
	\leq 2 \|\rho_\delta * V_0 - V_0\|_\infty
\]
since $V_\delta^\varepsilon\to V_\delta$ in $L^\infty\left((-M,M)\times[0,T]\right)$.
 Sending $\delta\downarrow0$, we obtain
\[
	\varlimsup_{\varepsilon\downarrow0} \sup_{0<t<T} \|V-V^\varepsilon\|(t) = 0
\]
since $V_0$ is uniformly continuous.

We next study the even part.
 The general strategy is the same but more involved than the odd part.
 For the even part, we first note that $V_\mathrm{even}^\varepsilon$ solves
\begin{empheq}[left = {\empheqlbrace \,}]{alignat = 2} 
    &\tau_1 V_t = V_{yy} - a^2(V-1), \quad y \in I_\varepsilon := \left( -L/\varepsilon, 0 \right), \quad t>0, \notag \\
    &V_y(0,t) + bV(0,t) = 0,\quad V_y \left(-L/\varepsilon, t \right) = 0, \quad t>0, \label{EBVP} \\
    &V|_{t=0}(y) = V_{0\mathrm{even}}^\varepsilon, \notag 
\end{empheq}
where $V_0^\varepsilon(y)=V^\varepsilon(y,0)$ for $y\in I_\varepsilon$.
 We suppress the word ``even'' from now on.
 We shall approximate $V_0^\varepsilon$ by a smoother function $V_{0\delta}^\varepsilon$ and approximate $V_0=\lim_{\varepsilon\to0}V_0^\varepsilon$ by a smoother function uniformly.
 There are many possible ways, and we rather like an abstract way.
 Let $BUC(I_\varepsilon)$ denote the space of all bounded uniformly continuous functions in $\overline{I_\varepsilon}$.
 It is a Banach space equipped with the norm
\[
	\lVert f \rVert_{\infty,\varepsilon} = \sup \left\{ \left|f(x)\right| \bigm| x \in I_\varepsilon \right\}.
\]
If $\varepsilon=0$, then $I_\varepsilon$ should be interpreted as $(-\infty,0]$.
 Let $A$ be the operator on $BUC(I_0)$ defined by
\[
	Af := (-\partial_y^2 + a^2)f \quad\text{(in the distribution sense)}
\]
with the
\[
	D(A) = \left\{ f \in BUC(I_0) \bigm| Af \in BUC(I_0),\ f_y(0) + bf(0) = 0 \right\}.
\]
A standard theory \cite{L} implies that $-A$ generates an analytic semigroup $e^{-tA}$ in $BUC(I_0)$.
 In particular,
\begin{align*}
	\left\lVert Ae^{-tA} f \right\lVert_{\infty,0} &\leq Ct^{-1} \lVert f \rVert_{\infty,0}, \\
	\left\lVert e^{-tA} f \right\lVert_{\infty,0} &\leq C \lVert f \rVert_{\infty,0}
\end{align*}
with some constant $C>0$ independent of time $t\in(0,1)$ and $f\in BUC(I_0)$.
 For a function $h\in BUC(I_\varepsilon)$, we extend it to $\widetilde{h}$ so that $\widetilde{h}(x)=h\left(-L/\varepsilon\right)$ for $x<-L/\varepsilon$.
 For $V_0$, we set $V_{0\delta}=e^{-\delta A}V_0$.
 For $V_0^\varepsilon$, we tempt to set $V_{0\delta}^\varepsilon=e^{-\delta A}\widetilde{V_0^\varepsilon}$.
 However, unfortunately, $V_{0\delta}^\varepsilon$ does not satisfy the boundary condition at $-L/\varepsilon$ though it satisfies $\left(\partial_y V_{0\delta}^\varepsilon+bV_{0\delta}^\varepsilon\right)(0)=0$ and is $C^2$ (actually smooth).
 We set $\sigma(x)=\rho_{\delta'}*\left(1-|x|\right)_+$ for a fixed $\delta'>0$ so that $\sigma'\left(1/2\right)=-\kappa$ with $\kappa<0$. 
 We set $\sigma_{\delta''}(y)=\delta''\sigma\left(y/\delta''\right)$ for small $\delta''>0$.
 For a given $h\in C^2(-\infty,0]$, we modify
\[
	h_{\delta''} (y) = h(y) + c\sigma_{\delta''}(y-\nu)
\]
where we take $c$ so that $c\kappa=h_y\left(-L/\varepsilon\right)$, $-L/\varepsilon-\nu=\delta''/2$.
 By this modification,
\[
	\partial_y h_{\delta''} \left(-L/\varepsilon\right) = 0
\]
and $h_{\delta''}\to h$ in $L^\infty(-\infty,0)$ as $\delta''\to0$.
 We set
\[
	V_{0\delta}^\varepsilon =\left(e^{-\delta A} \widetilde{V_0^\varepsilon} \right)_\delta.
\]
Since $V_{0\delta}^\varepsilon$ satisfies the boundary condition on the boundary of $I_\varepsilon$ and is smooth, we observe that $\partial_t V_\delta^\varepsilon$ is continuous up to the boundary of $I_\varepsilon\times[0,T)$, where $V_\delta^\varepsilon$ denotes the solution of \eqref{EBVP} with initial data $V_{0\delta}^\varepsilon$.
 By the maximum principle,
 \begin{align}
\begin{aligned}\label{EMAX} 
 	&\lVert V_\delta^\varepsilon -1 \rVert_{\infty,\varepsilon}(t) \leq \lVert V_{0\delta}^\varepsilon - 1\rVert_{\infty,\varepsilon} \\
 	 &\lVert \partial_y V_\delta^\varepsilon \rVert_{\infty,\varepsilon}(t) \leq \lVert \partial_y V_{0\delta}^\varepsilon \rVert_{\infty,\varepsilon} + b\lVert V_\delta^\varepsilon \rVert_{\infty,\varepsilon} \\
 	&\lVert \partial_t V_\delta^\varepsilon \rVert_{\infty,\varepsilon}(t) \leq \lVert \partial_t V_{0\delta}^\varepsilon \rVert_{\infty,\varepsilon},
\end{aligned}
\end{align}
where $\partial_tV_{0\delta}^\varepsilon$ for initial data $V_{0\delta}^\varepsilon$ should be interpreted as in the proof for the odd part.
 The term involving $b$ appears because of the Robin type boundary condition.
 As in the case for \eqref{ENoS2}, by the Arzel\`a--Ascoli theorem and the uniqueness of the limit equation, we can prove that $V_\delta^\varepsilon$ converges to $V_\delta$ locally uniformly in $\mathbf{R}\times[0,\infty)$.
 Note that for a fixed $\delta>0$, the right-hand sides of \eqref{EMAX} are uniformly bounded as $\varepsilon\to0$ since $V_0^\varepsilon$ converges to $V_0$ uniformly.
 The comparison principle implies that
\begin{align*}
	&\|V_\delta^\varepsilon - V^\varepsilon\|_{\infty,\varepsilon}(t) \leq \|V_{0\delta}^\varepsilon - V_0^\varepsilon\|_{\infty,\varepsilon},
	&\hspace{-5em} t > 0 \\
	&\|V - V_\delta\|_{\infty,0}(t) \leq \|V_{0\delta} - V_0\|_{\infty,0} 
	&\hspace{-5em} t > 0.
\end{align*}
Thus
\begin{align*}
	\|V - V^\varepsilon\|(t) 
	&\leq \|V - V_\delta\|(t) + \|V_\delta - V_\delta^\varepsilon\|(t) + \|V_\delta^\varepsilon - V^\varepsilon\|(t) \\
	&\leq \|V_0 - V_{0\delta}\|_{\infty,0} + \|V_\delta - V_\delta^\varepsilon\|(t) + \|V_{0\delta}^\varepsilon - V_0^\varepsilon\|_{\infty,\varepsilon},
\end{align*}
where the norm $\|\cdot\|$ is taken in $L^\infty(0,M)$ for $M>0$.
 Taking the supremum in $t\in(0,T)$ and sending $\varepsilon\to0$, we obtain that
\[
	\sup_{0<t<T} \| V - V^\varepsilon \|(t) \leq 2 \|V_{0\delta} - V_0 \|_\infty
\]
since we know $\sup_{0<t<T}\|V_\delta-V_\delta^\varepsilon\|(t)\to0$ as $\varepsilon\to0$.
 Sending $\delta\to0$, we conclude that $V^\varepsilon$ converges to $V$ locally uniformly in $[0,\infty)\times[0,\infty)$.

Since we know that $V^\varepsilon$ is continuous up to $y=0$ and $t=0$, this gives the local uniform convergence of $V^\varepsilon(0,t)$ in $[0,\infty)$.
 The proof is now complete.
\end{proof}

If we only assume that the initial data for \eqref{ENoS2} and \eqref{EBVP} is bounded and Lipschitz, we have a similar estimate in \eqref{EMAX} up to the first derivative of the solution.
However, the estimate for the time derivative should be altered.
Since we used such an estimate in Lemma~\ref{LGeq}, we state it in the case of $\varepsilon=0$ for the reader's convenience.

\begin{lemma}
	\label{Lemma3.1}
	Let $V$ be the bounded solution of \eqref{ERR} in $\mathbf{R}\times(0,\infty)$ with a bounded and Lipschitz continuous initial data $V_0$.
	Then for each $T>0$, there is a constant $C$ depending only on $a$, $b$, and $T$ such that 
	\begin{equation*}
		t^{1/2}\|\partial_tV\|_{L^\infty(\mathbf{R})}(t)\le C\left(\|\partial_yw_0\|_{L^\infty(\mathbf{R})}+\|w_0\|_{L^\infty(\mathbf{R})}+1\right)\quad\text{for}\quad t\in(0,T).
	\end{equation*}
\end{lemma}

\begin{proof}
	\quad
	We give direct proof.
	We may assume that $\tau_1=1$.
	We set $w=V-1$ and $u=e^{a^2t}w$ to get
	\begin{equation*}
		u_t=u_{xx}-2b\partial_x\{1_{x>0}\}\left(u+e^{a^2t}\right),
	\end{equation*}
	where we denote by $x$ instead of $y$.
	We consider this equation with initial data $w_0=V_0-1$.
	It suffices by simple scaling $u_\lambda(x,t)=u(\lambda x,\lambda^2t)$ to prove the desired estimate for some $T$ independent of $w_0$.

	Let $E(x,t)$ be the Gauss kernel as before.
	Then the solution can be represented as
	\begin{equation}
		u(x,t)=(E\ast_xw_0)(x,t)-\int_0^tE(x,t-\tau)h(\tau)\,d\tau,
		\quad
		h(t)=2b\left(u(0,t)+e^{a^2t}\right),
		\label{EINT}
	\end{equation}
	where
	\begin{equation*}
		(E\ast_xw_0)(x,t)\coloneqq\int_{-\infty}^\infty E(x-y,t)w_0(y)\,dy.
	\end{equation*}
	Since we can approximate a smooth $w_0$, establishing
	\begin{equation}
		\|\partial_tu\|_{L^\infty(\mathbf{R})}(t)
		\le Ct^{-1/2}\left(\|\partial_xw_0\|_{L^\infty(\mathbf{R})}+\|w_0\|_{L^\infty(\mathbf{R})}+1\right)\quad\text{for}\quad t\in(0,T)
		\label{EKEYE}
	\end{equation}
	with some positive constants $C$ and $T$ independent of $w_0$ suffices, assuming that $\partial_tu$ exists and is bounded in $\mathbf{R}\times(0,T)$ for small $T$.
	By the maximum principle \eqref{EMAX} and the corresponding estimate for the odd part, we know that
	\begin{equation}
		\|u\|_{L^\infty(\mathbf{R})}(t)\le c\left(1+\|w_0\|_{L^\infty(\mathbf{R})}\right)
		\label{EEE1}
	\end{equation}
	with $c$ independent of $w_0$ and $t>0$.
	We estimate $\partial_tu$ in \eqref{EINT}.
	Since $\|f\ast_xg\|_{L^\infty(\mathbf{R})}\le\|f\|_{L^1(\mathbf{R})}\|g\|_{L^\infty(\mathbf{R})}$ and
	\begin{equation*}
		\partial_t(E\ast_xw_0)=(\partial_xE)\ast_x\partial_xw_0,
	\end{equation*}
	we easily see (cf.~\cite[Chapter~1]{GGS}) that
	\begin{equation} \label{EEE2}
		\left\|\partial_t (E *_x w_0)\right\|_{L^\infty(\mathbf{R})}(t)
		\leq \frac{c'}{t^{1/2}} \|\partial_x w_0\|_{L^\infty(\mathbf{R})}
		\quad\text{for}\quad t > 0
	\end{equation}
	with $c'$ independent of $w_0$.
	The second term of the right-hand side of \eqref{EINT} is more involved than the first term because $h$ contains $u$.
	We observe that
	\begin{equation*}
		\partial_t \int_0^t E(x,t-\tau) h(\tau)\, d\tau
		= \int_0^t E(x,t-\tau) \partial_\tau h(\tau)\, d\tau
		+ E(x,t) h(0).
	\end{equation*}
	
	Since $|E|\le(4\pi t)^{-1/2}$, it holds that
	\begin{equation*}
		\left|
			\int_0^tE(x,t-\tau)\partial_\tau h(\tau)\,d\tau
		\right|
		\le
		\int_0^t\frac{1}{(4\pi(t-\tau))^{1/2}}\frac{1}{\tau^{1/2}}\,d\tau\cdot\sup_{0<t<T}|t^{1/2}\partial_th(t)|
	\end{equation*}
	for $t\in(0,T)$.
	
	Thus,
	\begin{align*}
		\sup_{0<t<T} \left\| \partial_t \int_0^t E(x,t-\tau) h(\tau)\, d\tau \right\|_\infty
		&\leq \,C_1 \sup_{0<t<T} \left| t^{1/2} \partial_t h(t) \right| + (4\pi t)^{-1/2} \left|h(0)\right| \\
		&\hspace{-100pt}\leq \,C_2 \left( \sup_{0<t<T} \| t^{1/2} \partial_t u \|_{L^\infty(\mathbf{R})} (t) + 1 \right) + C_3 t^{-1/2} \left( \| u \|_{L^\infty(\mathbf{R})} (t) + 1 \right)
	\end{align*}
	with some constants $C_j$ ($j=1,2,3$).
	By \eqref{EINT} and \eqref{EEE2}, we now observe that
	\begin{multline*}
		\sup_{0<t<T} \left\| t^{1/2} \partial_t u \right\|_{L^\infty(\mathbf{R})} (t)
		\leq C_4 \left\| \partial_x w_0 \right\|_{L^\infty(\mathbf{R})}\\
		+ C_2 T^{1/2} \left( \sup_{0<t<T} \| t^{1/2} \partial_t u \|_{L^\infty(\mathbf{R})} + 1 \right) + C_3 \sup_{0<t<T} \left(\|u\|_{L^\infty(\mathbf{R})}(t) + 1 \right)
	\end{multline*}
	with $C_4$ independent of $w_0$, $u$, and $T$.
	Applying estimate for $\|u\|_\infty$ in \eqref{EEE1}, we conclude \eqref{EKEYE} for sufficiently small $T$ by absorbing $C_2T^{1/2}\sup_{0<t<T}\|t^{1/2}\partial_tu\|_{L^\infty(\mathbf{R})}$ in the right-hand side to the left.
\end{proof}

\section{Dirichlet condition for the total variation flow} \label{S4} 

In this section, we recall a notion of total variation flow for a given $v$ and prove Lemma \ref{LST}.
 We consider
\begin{equation} \label{ETOTV}
	\partial_t u = \operatorname{div} \left(\beta \nabla u/|\nabla u|\right)
	\quad\text{in}\quad I \times (0, T),
\end{equation}
where $\beta\in C\left(I\times[0,T]\right)$ is a given nonnegative function;
 here, $I=(p_0,p_1)$ is an open interval and $T>0$.
 If we impose the Dirichlet boundary condition
\begin{equation} \label{EDir}
	u = g \quad\text{on}\quad \partial I,
\end{equation}
\eqref{ETOTV} with \eqref{EDir} should be interpreted as an $L^2$-gradient flow of a time-dependent total variation type energy
\[
	\Phi^t(u) = \int_I \beta(x,t) |u_x| 
	+ \sum_{i=0}^1 |\gamma u-g| (p_i)\beta(p_i,t)
\]
when $\int\beta|u_x|$ is a weighted total variation of $\beta$ and $\gamma u$ is a trace of $u$ on $\partial I$.
 We consider this energy in $L^2(I)$ by $\Phi^t(u)=\infty$ when $\Phi^t(u)$ is not finite.
 It is clear that $\Phi^t$ is convex in $L^2(I)$.
 If $\beta$ is spatially constant, it is well known that $\Phi^t$ is also lower semicontinuous;
 see, e.g.\ \cite{ACM}.
 The solution of \eqref{ETOTV} with \eqref{EDir} should be interpreted as the gradient flow of form
\begin{equation} \label{EAGF}
	u_t \in - \partial \Phi^t(u),
\end{equation}
where $\partial\Phi^t$ denotes the subdifferential of $\Phi^t$ in $L^2(I)$, i.e.,
\[
	\partial\Phi^t(u) = \left\{ f \in L^2(I) \biggm|
	\Phi^t(u+h) - \Phi^t(u)
	\geq \int_I hf\,dx\ \text{for all}\ h \in L^2(I) \right\}.
\]
It is standard that \eqref{EAGF} is uniquely solvable for given initial data $u_0\in L^2(I)$ if $\Phi^t$ does not depend on time and is lower semicontinuous and convex on the Hilbert space $L^2(I)$ (see, for instance, \cite{Kom,Br}).
 It applies to the total variation flow case when $\beta$ is a content.
 In this case, the subdifferential becomes
\[
	\partial\Phi(u) = \left\{ v \in L^2(I) \biggm|
	v = -(\beta z)_x,\ \|z\|_\infty \leq 1, \ \Phi(u) = \int_I uv\,dx \right\},
\]
when $\Phi=\Phi^t$; see \cite{ACM}.
 The equation \eqref{EAGF} is
\[
	u_t = \operatorname{div} \beta z
\]
with $|z|\leq1$ in $I$ and $\Phi(u)=\int_I u\operatorname{div}\beta z\,dx$.
 We mimic this notion of the solution.
 A function $u\in C\left([0,T),L^2(I)\right)$ is a solution to \eqref{ETOTV} with \eqref{EDir} if there is $z\in L^\infty\left(I\times(0,T)\right)$ such that
\begin{align}
	& u_t = (\beta z)_x \quad\text{in}\quad I \times (0,T) \label{ECah1} \\
	& |z| \leq 1 \quad\text{in}\quad I \times (0,T) \label{ECah2} \\
	& - \int_I u (\beta z)_x\, dx = \int_I \beta |u_x|
	+ \sum_{i=0}^1 |\gamma u-g| \beta(p_i,t). \label{ECah3}
\end{align}
Under this preparation, we shall prove Lemma \ref{LST}.
\begin{proof}[Proof of Lemma \ref{LST}]
\quad
Since $u_t^b=0$, \eqref{ECah1} says that $\beta z$ is a constant $c$.
 The condition \eqref{ECah2} is equivalent to saying that $|c|\leq\beta$.
 Since
\[
	\int_{-L}^L u(\beta z)_x\, dx
	= \int_0^L b(\beta z)_x\, dx = -\beta(0) z(0)b,
\]
\eqref{ECah3} is equivalent to
\[
	\beta(0) z(0) b = \beta(0)b.
\]
In other words, $c$ must be $\beta(0)$.
Thus, the existence of $z$ satisfying $|z|\leq1$ is guaranteed if and only if
\[
	c = \beta(0) \leq \beta(x)
	\quad\text{for all}\quad x \in (-L,L).
\]
The equation \eqref{ECah3} is fulfilled by taking $g=b$.
 Thus $u^b$ is a stationary solution to \eqref{EG2} with \eqref{EDi}.
\end{proof}
%

\section{Numerical experiment} \label{S5} 

In this section, we calculate the solution of \eqref{ER}--\eqref{ERI} with $v_0(x)=1-ce^{-a|x|/\varepsilon}$ and compare its value at $x=0$ with an explicit solution of \eqref{EGFD} whose explicit form is given in \eqref{EEXPS}.

\subsection{Numerical scheme}

Since the initial function, $v_0$, is an iven function, the original problem  \eqref{ER}--\eqref{ERI} is reduced to 
\begin{align*}
	\begin{dcases*}
		\frac{\tau_1}{\varepsilon}v_t=\varepsilon v_{xx}-\frac{a^2(v-1)}{\varepsilon}&in $(0,L)\times(0,\infty)$,\\
		-v_x(0,t)+\frac{b}{\varepsilon}v(0,t)=0,&for $t>0$,\\
		v_x(L,t)=0,&for $t>0$,\\
		v(x,0)=v_0(x),&for $x\in[0,L]$.
	\end{dcases*}
\end{align*}

The computational region $[-L,L]$ is divided into unfiorm mesh partitions:
\begin{align*}
	x_i=i\Delta x,
	\quad
	i=-1,0,\ldots,N,N+1,
	\quad
	\Delta x=\frac{L}{N}.
\end{align*}
The points $x_{-1}$ and $x_{N+1}$ are needed to handle the Neumann boundary conditions.

The approximation of $v$ at $x=x_i$ is written as $v_i$.
The central finite difference approximates the Laplace operator, and the time derivative is approximated by the backward difference, yielding the folliwng linear system
\begin{align*}
	\frac{\tau_1}{\varepsilon}\frac{v_i-\hat{v}_i}{\Delta t}=\varepsilon\frac{v_{i-1}-2v_i+v_{i+1}}{(\Delta x)^2}-\frac{a^2(v_i-1)}{\varepsilon},\quad i=0,1,\ldots,N,
\end{align*}
where $\hat{v}_i$ is the value known at the current time, and $v_i$ is the value to be found at the next time.
The Neumann boundary conditions at $x=0$ and $x=L$ are approximated as
\begin{align*}
	-\frac{v_1-v_{-1}}{2\Delta x}+\frac{b}{\varepsilon}v_0=0
	\quad\text{and}\quad
	\frac{v_{N+1}-v_{N-1}}{2\Delta x}=0
\end{align*}
by the central finite differences.

\subsection{Results}

Some results are shown for different values of $c$ for parameters
\begin{align*}
	L=1,
	\quad
	a=1,
	\quad
	N=200,
	\quad
	\Delta t=(\Delta x)^2,
	\quad
	b=1.
\end{align*}
The results of the numerical experiments are summarised in Figure~\ref{fig:resutls}.
\begin{figure}[tb]
	\begin{minipage}{.5\hsize}
		\centering
		\includegraphics[width=\hsize]{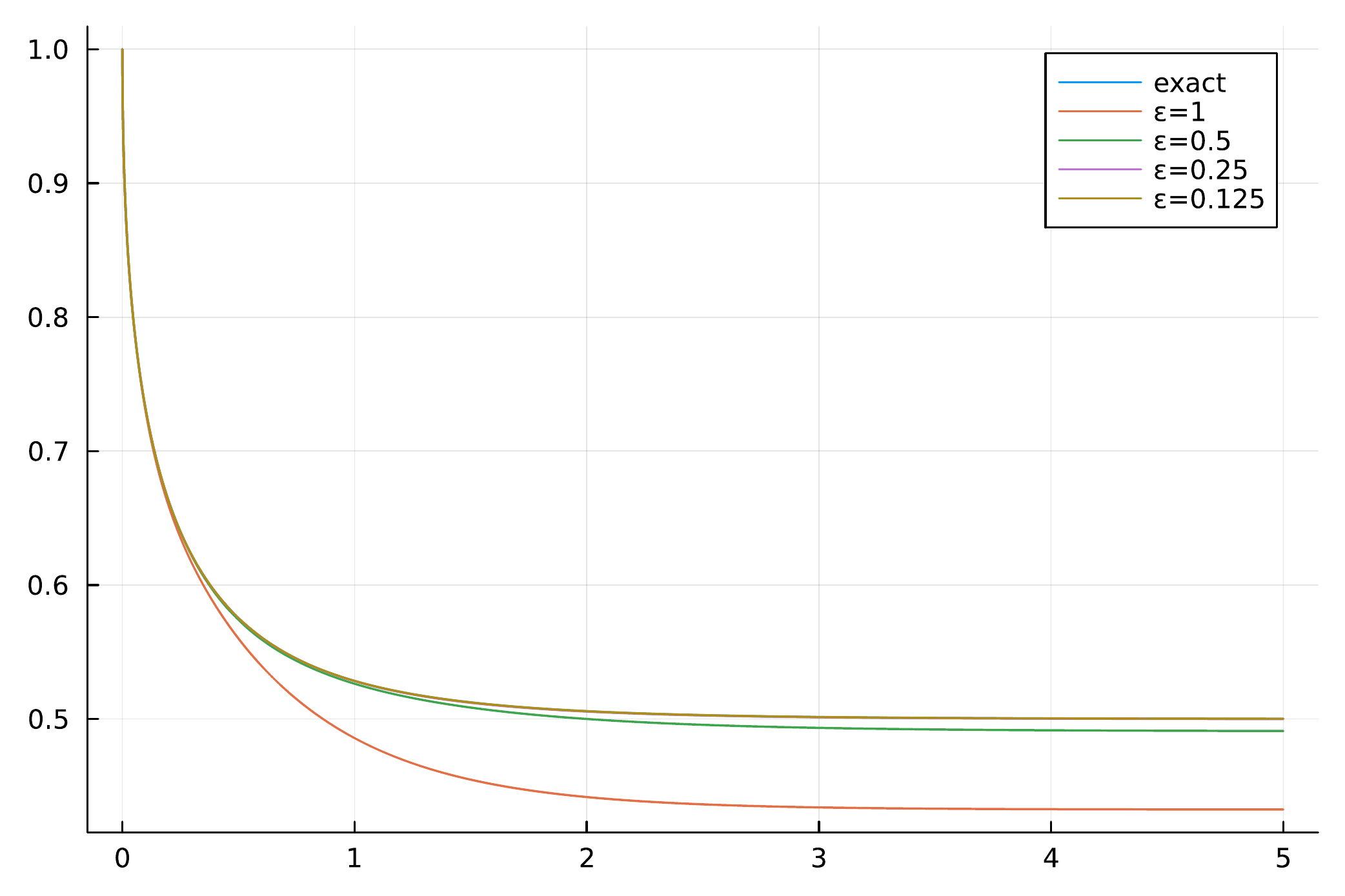}

		\subcaption{}
	\end{minipage}%
	\begin{minipage}{.5\hsize}
		\centering
		\includegraphics[width=\hsize]{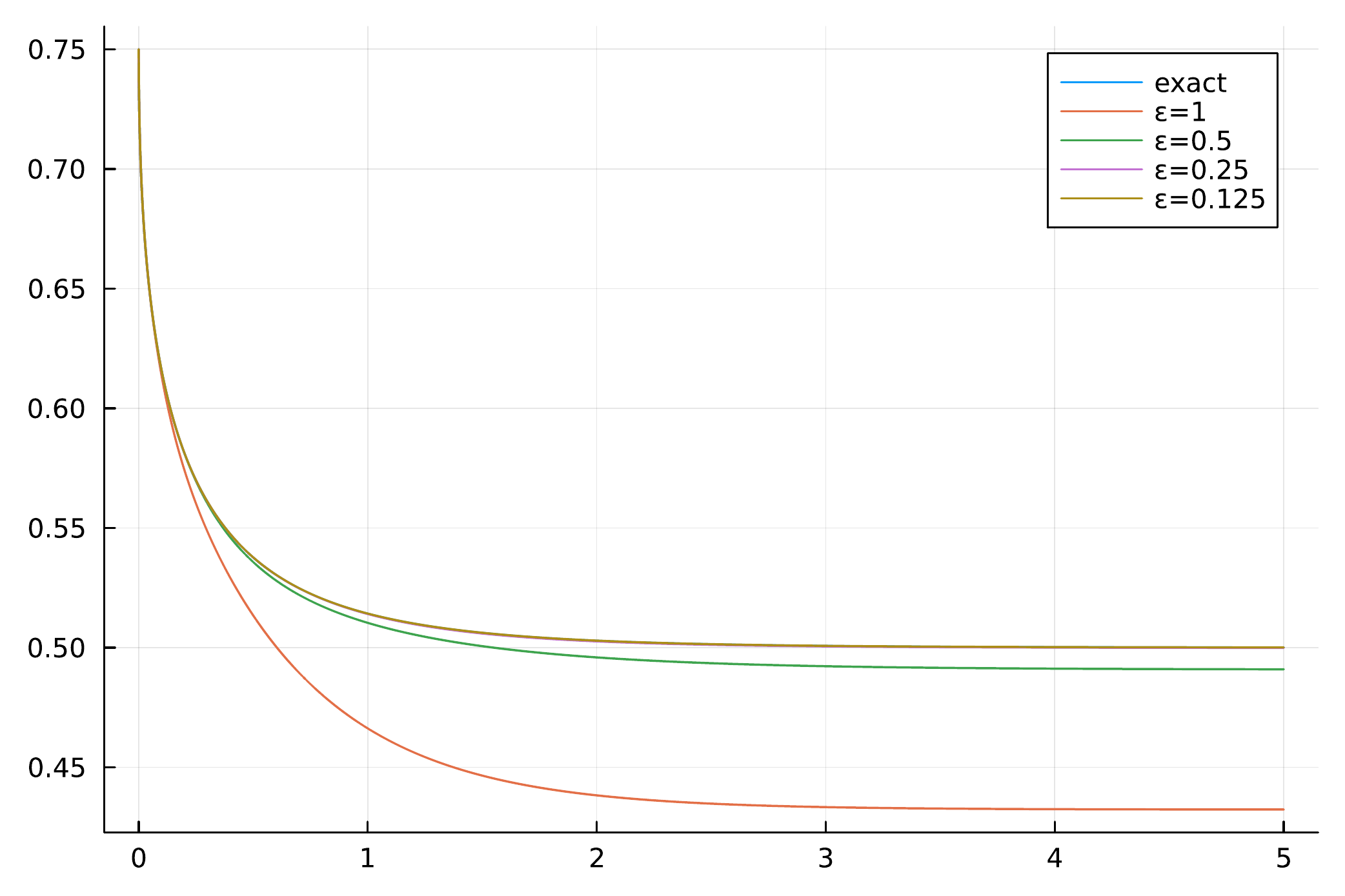}

		\subcaption{}
	\end{minipage}%

	\begin{minipage}{.5\hsize}
		\centering
		\includegraphics[width=\hsize]{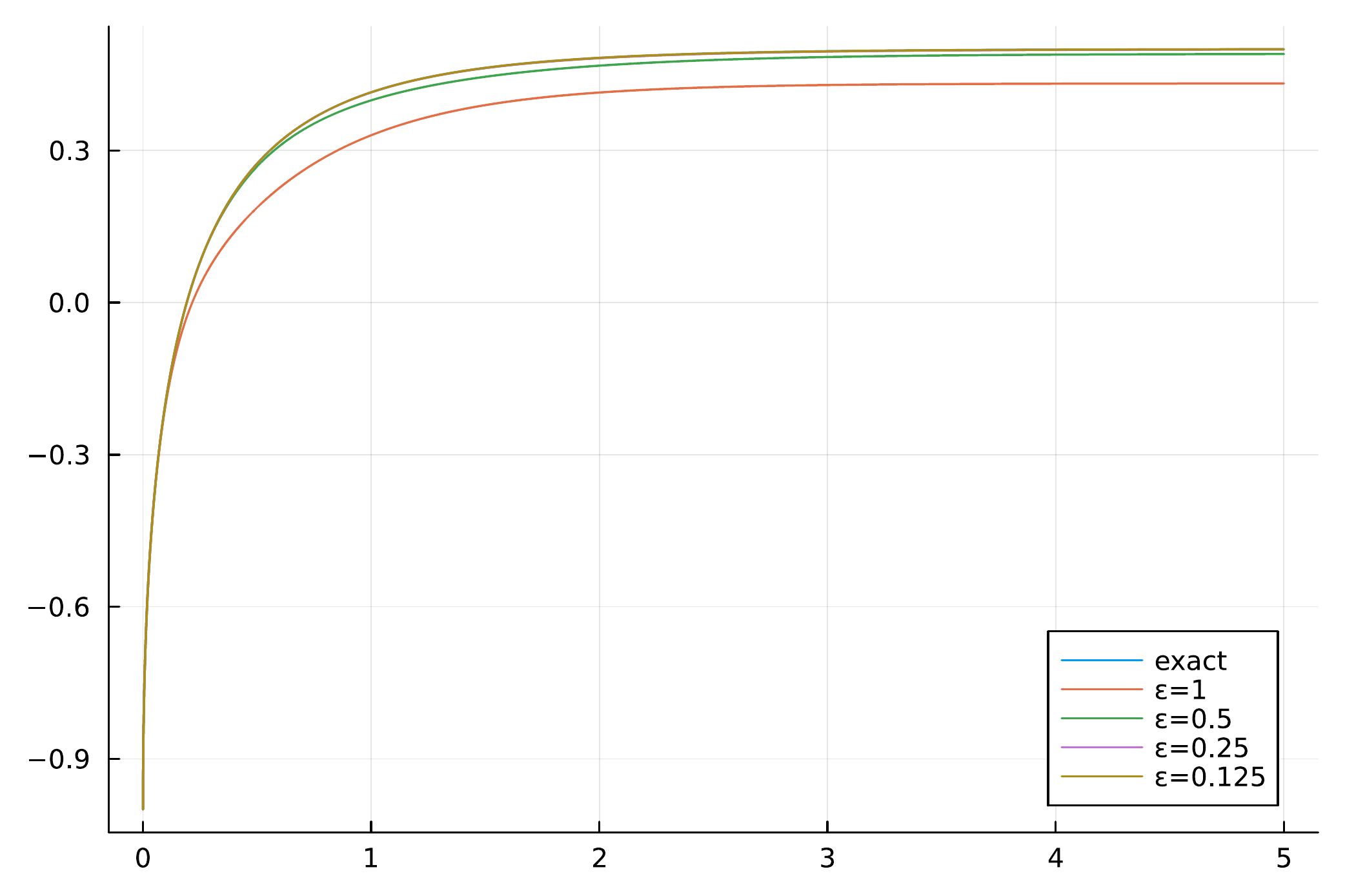}

		\subcaption{}
	\end{minipage}%
	\begin{minipage}{.5\hsize}
		\centering
		\begin{tabular}{|c|c|c|c|}\hline
			\diagbox{$\varepsilon$}{$c$}&$0$&$1/4$&$2$\\\hline\hline
			$1$&$6.8\cdot10^{-2}$&$6.8\cdot10^{-2}$&$6.7\cdot10^{-2}$\\\hline
			$2^{-1}$&$9.1\cdot10^{-3}$&$9.1\cdot10^{-3}$&$9.3\cdot10^{-3}$\\\hline
			$2^{-2}$&$1.4\cdot10^{-4}$&$1.5\cdot10^{-4}$&$2.6\cdot10^{-4}$\\\hline
			$2^{-3}$&$4.6\cdot10^{-5}$&$4.6\cdot10^{-5}$&$4.6\cdot10^{-5}$\\\hline
		\end{tabular}

		\subcaption{}
	\end{minipage}%
	\caption{Results of numerical experiments: (a) $c=0$, (b) $c=1/4$, (c) $c=2$, (d) table of $L^\infty$-errors for different $\varepsilon$ and $c$.}
	\label{fig:resutls}
\end{figure}
In (a), (b), and (c), the horizontal axis represents time, and the vertical axis represents the value at the origin.
As $\varepsilon$ is decreased, the numerical solution converges to the exact solution to the extend that the exact and numerical solutions overlap.
Indeed, the table of $L^\infty$-errors for different values of $\varepsilon$ and $c$ is shown in (d).
The errors for $\varepsilon=2^{-3}$ are of order $10^{-5}$, indicating that the solution for small $\varepsilon$ is an excellent approximation to the solution of the fractional time differential equation obtained as the singular limit.

\section*{Acknowledgements}

The work of the first author was partly supported by JSPS KAKENHI Grant Numbers JP19H00639 and JP20K20342, and by Arithmer Inc., Ebara Corporation,\ and Daikin Industries, Ltd.\ through collaborative grants.
The work of the third author was partly supported by JSPS KAKENHI Grant Number JP20K20342.
The work of the fifth author was partly supported by JSPS KAKENHI Grant Numbers JP22K03425, JP22K18677, 23H00086.

%
%


\end{document}